\documentclass[10pt,reqno]{amsart}
\usepackage{url}
\usepackage{amsthm}
\usepackage{fullpage}
\usepackage{amsmath}
\usepackage{amssymb}
\usepackage{amsfonts}

\usepackage{enumitem}
\usepackage{mathscinet}
\usepackage{lipsum}
\usepackage[english]{babel}
\usepackage[autostyle]{csquotes}
\usepackage{bbold}

\usepackage{graphicx}

\usepackage{hyperref}

\hypersetup{
    colorlinks=true,
    linkcolor=red,
    citecolor=blue,
    }

\usepackage[utf8]{inputenc}
\usepackage[english]{babel}
 
\usepackage{color}
\usepackage{comment}

\newtheorem{thm}{Theorem}[section]
\newtheorem{definition}[thm]{Definition}
\newtheorem{theorem}[thm]{Theorem}
\newtheorem*{theorem*}{Theorem}
\newtheorem*{oseledec*}{Oseledec Theorem}
\newtheorem{cor}[thm]{Corollary}
\newtheorem*{cor*}{Corollary}
\newtheorem{claim}[thm]{Claim}
\newtheorem{lemma}[thm]{Lemma}
\newtheorem*{lemma*}{Lemma}
\newtheorem{prop}[thm]{Proposition}

\makeatletter
\def\moverlay{\mathpalette\mov@rlay}
\def\mov@rlay#1#2{\leavevmode\vtop{%
   \baselineskip\z@skip \lineskiplimit-\maxdimen
   \ialign{\hfil$\m@th#1##$\hfil\cr#2\crcr}}}
\newcommand{\charfusion}[3][\mathord]{
    #1{\ifx#1\mathop\vphantom{#2}\fi
        \mathpalette\mov@rlay{#2\cr#3}
      }
    \ifx#1\mathop\expandafter\displaylimits\fi}
\makeatother

\newcommand{\bigcupdot}{\charfusion[\mathop]{\bigcup}{\cdot}}
\let\ul\underline
\let\wh\widehat
\let\wt\widetilde

\usepackage{wasysym}

\newcommand{\wpi}{\wh{\pi}}
\newcommand{\cont}{\text{Grassmann-H\"older continuous}}

\def\tigma{\ensuremath{\widetilde{\sigma}_R}}
\def\Tigma{\ensuremath{\widetilde{\Sigma}}}
\def\Jac{\ensuremath{\mathrm{Jac}}}
\def\HWT{\ensuremath{\mathrm{RWT}}}
\def\WT{\ensuremath{\mathrm{WT}}}

\def\CAR{\ensuremath{\curvearrowright}}
\def\Sig{\ensuremath{\widehat{\Sigma}}}
\title{
Invariant Family of Leaf measures and The Ledrappier-Young Property for Hyperbolic Equilibrium States}
\author{Snir Ben Ovadia}
\begin{document}
\maketitle
\begin{abstract}
$M$ is a Riemannian, boundaryless, and compact manifold with $\dim M\geq 2$, and $f$ is a $C^{1+\beta}$ ($\beta>0$) diffeomorphism of $M$. $\varphi$ is a H\"older continuous potential on $M$. We construct an invariant and absolutely continuous 
family of measures (with transformation relations defined by $\varphi$), which sit on local unstable leaves. We present two main applications. First, given an ergodic homoclinic class $H_\chi(p)$, we prove that $\varphi$ admits a local equilibrium state on $H_\chi(p)$ if and only if $\varphi$ is ``recurrent on $H_\chi(p)$" (a condition tested by counting periodic points), and one of the leaf measures gives a positive measure to a set of positively recurrent hyperbolic points; and if an equilibrium measure exists, the said invariant and absolutely continuous family of measures constitutes as its conditional measures. An immediate corollary is the local product structure of hyperbolic equilibrium states. 
Second, we prove a Ledrappier-Young property for hyperbolic equilibrium states- if $\varphi$ admits a conformal family of leaf measures, and a hyperbolic local equilibrium state, then the  leaf measures of the invariant family (respective to $\varphi$) are equivalent to the conformal measures (on a full measure set). This extends the celebrated result by Ledrappier and Young for hyperbolic SRB measures, which states that a hyperbolic equilibrium state of the geometric potential (with pressure 0) has conditional measures on local unstable leaves which are absolutely continuous w.r.t the Riemannian volume of these leaves.
\end{abstract}

\tableofcontents

\section{Introduction}\label{introintro}

When studying a complex system of many states and its dynamics, one approach is to find a measure on the space of all states of the system which is invariant under its dynamics, to represent a distribution of possible events when the system is in equilibrium. The collection of all possible such invariant measures could be very big, and the field of thermodynamic formalism offers canonical ways to single out specific invariant measures of importance. 

One way to single out measures of importance, is by the notion of equilibrium states: probability measures which maximize a dynamic quantity w.r.t some potential (function) on the system. More explicitly, let $M$ be the space of states of the system (a compact boundaryless Riemannian manifold), let $f\in\mathrm{Diff}^{1+}(M)$ be the dynamics of the system, and let $\varphi:M\rightarrow \mathbb{R}$ be a bounded measurable potential on $M$; then an equilibrium state of $\varphi$ maximizes $\mu\mapsto h_{\mu}(f)+\int_M \varphi d\mu$, where $h_{\mu}(f)$ is the metric entropy of $\mu$ w.r.t $f$. 

Two important instances of this notion are measures of maximal entropy (when $\varphi\equiv 0$), and SRB measures (when $\varphi$ is the geometric potential, i.e minus the logarithm of the Jacobian of $f$ when restricted to the dynamics on unstable leaves). The fact that equilibrium states of the geometric potential (when $\sup\{h_{\mu}(f)+\int \varphi d\mu\}=0$) are SRB measures is highly non trivial, and it is due to Ledrappier and Young, \cite{LedrappierYoungI}. It is an instance of their more general formula to compute the entropy of general measures, via local dimension of conditional measures on unstable leaves (see also \cite{LedrappierYoungII}). When $f$ preserves a smooth measure, the entropy formula is due to Pesin, \cite{Pesin77}.

SRB measures are measures which have conditional measures which are absolutely continuous w.r.t the induced Riemannian volume on unstable leaves, when disintegrated by a measurable partition subordinated to the foliation of unstable leaves (see Rokhlin's disintegration theorem). SRB stands for Sinai, Ruelle, and Bowen, due to their pioneering work in the field. Hyperbolic and ergodic SRB measures are of physical importance, as the set of points which satisfy the forward-time Birkhoff's ergodic theorem are of positive Riemannian volume (and thus, by convention, are observable in an experiment or a simulation). This property is due to Pesin's absolute continuity theorem (see \cite{Pesin77,KS,PughAndShub}). Tsujii showed a converse, a condition via Lyapunov exponents on orbits such that a physical measure must be an ergodic hyperbolic SRB, see \cite{Tsuji}.

While SRB measures are important, it is not clear in what generality systems admit them. Sinai had shown their existence in Anosov systems, and Bowen and Ruelle in Axiom A systems. The dual characterization through conditional leaf measures, and as an equilibrium state, by Ledrappier and Young, offered a new way to test if a system admits an SRB measure. 

Singling out measures of importance through their conditional leaf measures, is another important tool in thermodynamic formalism, in addition to the notion of equilibrium states. Some important examples of this are the Margulis construction for the measure of maximal entropy through its leaf measures (\cite{MargulisLeafMeasures}), and the extension of the horocyclic flow to high dimension variable negative curvature manifolds (see \cite{Marcus75,BowenMarcus77}). Unstable leaves can be viewed as the equivalence class of an equivalence relation on orbits. The notion of defining measures by conditional measures on unstable leaves, has been further generalized and studied as an equivalence class of orbits on manifolds (see \cite{Series80}) and symbolic spaces (see \cite{Series80} and \cite{PetersenSchmidt97}). 

Given a H\"older continuous potential, the idea of an equilibrium state is well defined, but it is not clear how to generally define the leaf measures corresponding to a potential. The Margulis leaf measures are conformal- i.e when pushed forward they remain invariant up to a factor of $e^{-h_{\mathrm{top}}(f)}$, where $h_{\mathrm{top}}(f)=\sup\{h_{\mu}(f):\mu\text{ is an invariant probability measure}\}$. Similarly, the conditional leaf measures of SRB measures are the induced Riemannian volume, which gains a factor of $\exp\Big($geometric potential$\Big)$. We therefore are looking for a family of conditional leaf measures which are conformal w.r.t our potential. More explicitly, if $\varphi$ is a bounded measurable potential, and $V^u$ is a local unstable leaf, we wish to find measures which satisfy $m_{V^u}\circ f^{-1}=e^{\varphi-P(\varphi)}\cdot m_{f[V^u]}$, where $P(\varphi)$ is the pressure of $\varphi$ (the quantity which an equilibrium state maximizes), which acts as a calibrating factor. In \cite{CPZuniformlyHyp,CPZ}, Climenhaga, Pesin, and Zelerowicz give a construction using Carath\'eodory dimension. Their construction works for a general H\"older continuous potential, in the partially hyperbolic setup with some restrictions on the central stable foliation and on the transitivity of the system (see the remark after Definition \ref{Cara} for more details).

The result of Ledrappier and Young connects the two approaches of singling out measures of importance as equilibrium states and through conditional measures on unstable leaves. This naturally raises the question of extending their result for more potentials than just the geometric potential. Given a 
$\cont$ potential, we give a construction of an invariant and absolutely continuous family of leaf measures which give a necessary and sufficient condition for the existence of a hyperbolic local equilibrium state via a leaf condition, and which act as the conditional leaf measures of the equilibrium state when it exists. We prove a Ledrappier-Young property, such that if the potential admits a hyperbolic local equilibrium state, and a conformal system of measures, then the conditional measures of the equilibrium state are equivalent to the conformal measures on a set of full measure, and in particular the conformal measures give a positive measure to the hyperbolic points. For the definitions of a hyperbolic local equilibrium state, the leaf condition, and a $\cont$ potential, see \textsection \ref{HWTsection}. In particular, Grassmann-H\"older continuity applies to the family of potentials $\{t\cdot \varphi\}_{t\in[0,1]}$, where $\varphi$ is the geometric potential, and this family varies continuously between the $0$ potential and the geometric potential.

\section{Basic Definitions and Main Results}\label{HWTsection}
 
 Throughout this paper $M$ is a Riemannian, boundaryless, and compact manifold of dimension $d\geq 2$, and $f\in \mathrm{Diff}^{1+\beta}(M), \beta>0$. Fix $\chi>0$. 
 \subsection{Basic Definitions}
\begin{definition}[$\chi$-summable points]\label{ChiHyp} A point $x\in M$ is called {\em $\chi$-summable} if there is a unique splitting $T_xM=H^s(x)\oplus H^u(x)$ s.t
         \begin{align*}
        &\sup_{\xi_s\in H^s(x),|\xi_s|=1}\sum_{m=0}^\infty|d_xf^m\xi_s|^2e^{2\chi m}<\infty,\sup_{{\xi_u\in H^u(x),|\xi_u|=1}}\sum_{m=0}^\infty|d_xf^{-m}\xi_u|^2e^{2\chi m}<\infty
        .
    \end{align*}
Let {\em $\chi$-summ} define the set of $\chi$-summable points. An $f$-invariant probability measure carried by $\chi$-$\mathrm{summ}$ is called {\em $\chi$-hyperbolic}.
\end{definition}


For each $x \in\chi\text{-}\mathrm{summ}$, write $s(x):=\mathrm{dim}(H^s(x)),u(x):=\mathrm{dim}(H^u(x))$. Notice, $\chi$-summ is $f$-invariant, and $H^s(\cdot), H^u(\cdot)$ are invariant under $d_\cdot f$.
\begin{theorem}[Pesin-Oseledec~Reduction~Theorem]\label{pesinoseledec}
For each point $x \in\chi\text{-}\mathrm{summ}$, there exists an invertible linear map $C_\chi(x):\mathbb{R}^d\rightarrow T_xM$, which depends measurably on $x$, such that 
$C_\chi(x)[\mathbb{R}^{s(x)}\times\{0\}]=H^s(x),C_\chi(x)[\{0\}\times\mathbb{R}^{u(x)}]=H^u(x)$. 
$C_\chi(\cdot)$ are chosen measurably on $\chi\text{-}\mathrm{summ}$, and the choice is unique up to a composition with an orthogonal mapping of the `` stable" and of the ``unstable" subspaces of the tangent space. In addition,
$$
C_\chi^{-1}(f(x))\circ d_xf\circ C_\chi(x)=\begin{pmatrix}D_s(x)  &   \\  & D_u(x)
\end{pmatrix},
$$
where $D_s(x),D_u(x)$ are square matrices of dimensions $s(x),u(x)$ respectively, and $\|D_s(x)\|,\|D_u^{-1}(x)\|\leq e^{-\chi}$,$\|D_s^{-1}(x)\|,\|D_u(x)\|\leq \kappa$ for some constant $\kappa=\kappa(f,\chi)>1$.
\end{theorem}
The Pesin-Oseledec reduction theorem has many different versions, which are suitable for different setups. We use the version which appears, with proof, in \cite[Theorem~2.4]{SBO}.

It is possible to show that $\|C_\chi^{-1}(x)\|^2=\sup\limits_{\overset{|\xi_s+\xi_u|=1,}{\xi_s\in H^s(x), \xi_u\in H^u(x)}}\left\{2\sum\limits_{m\geq0}|d_xf^m\xi_s|^2e^{2\chi m}+ 2\sum\limits_{m\geq0}|d_xf^{-m}\xi_u|^2e^{2\chi m}\right\}$. In addition, $\|C_\chi(\cdot)\|\leq 1$. See \cite[Theorem~2.4, Lemma~2.9]{SBO} for details. $\|C_\chi^{-1}(x)\|$ serves a measurement of the hyperbolicity of $x$- the greater the norm, the worse the hyperbolicity (i.e slow contraction/expansion on stable/unstable spaces, or small angle between the stable and unstable spaces).

\begin{definition}\label{littleQ} Let $\epsilon>0$, and let $x \in\chi\text{-}\mathrm{summ}$, then

\medskip
    $$Q_\epsilon(x):=\max\left\{Q\in \{e^{\frac{-\ell\epsilon}{3}}\}_{\ell\in\mathbb{N}}:Q\leq \frac{1}{3^\frac{6}{\beta}}\epsilon^\frac{90}{\beta}\|C^{-1}_\chi(x)\|^\frac{-48}{\beta}\right\}.$$
\end{definition}

\begin{definition}[Recurrent $\epsilon$-weak temperability]\label{temperable}
Let $\epsilon>0$. A point $x\in\chi\text{-}\mathrm{summ}$ is called {\em $\epsilon$-weakly temperable} if  $\exists q:\{f^n(x)\}_{n\in\mathbb{Z}}\rightarrow \{e^{\frac{-\ell\epsilon}{3}}\}_{\ell\in\mathbb{N}}$ s.t
\begin{enumerate}
	\item $\frac{q\circ f}{q}=e^{\pm\epsilon}$,
	\item $\forall n\in\mathbb{Z}$, $q(f^n(x))\leq Q_\epsilon(f^n(x))$.
\end{enumerate}
If in addition to (1),(2), $q:\{f^n(x)\}_{n\in\mathbb{Z}}\rightarrow \{e^{-\frac{\ell \epsilon}{3}}\}_{\ell\in\mathbb{N}}$ can be chosen to also satisfy,

\begin{enumerate}
\item[(3)]$\limsup\limits_{n\rightarrow\pm\infty}q(f^n(x))>0$,
\end{enumerate}
then we say that $x$ is {\em recurrently $\epsilon$-weakly temperable}.
A function $q:\{f^n(x)\}_{n\in\mathbb{Z}}\rightarrow \{e^{\frac{-\ell\epsilon}{3}}\}_{\ell\in\mathbb{N}}$ which satisfies (1),(2), is called {\em an $\epsilon$-weakly tempered kernel of $x$}. Define $\WT_\chi^\epsilon:= \left\{x\in\chi\text{-}\mathrm{summ}:x\text{ is } \epsilon\text{-weakly temperable}\right\}$, and  $\HWT_\chi^\epsilon:= \left\{x\in\chi\text{-}\mathrm{summ}:x\text{ is recurrently } \epsilon\text{-weakly temperable}\right\}$.
\end{definition}

\noindent\textbf{Remark:} Every point $x\in\WT_\epsilon^\chi$ (for sufficiently small $\epsilon>0$) admits a global unstable manifold $W^u(x)=\{y\in M: d(f^{-n}(x),f^{-n}(y))\xrightarrow[n\rightarrow \infty]{}0\}$, and a global stable manifold $W^s(x)=\{y\in M: d(f^{n}(x),f^{n}(y))\xrightarrow[n\rightarrow \infty]{}0\}$. See \cite{BrinStuck}[\textsection5.6] for more details. In Definition \ref{globalstableunstable} we give an alternative (yet equivalent) notion of global stable/unstable manifolds.

\begin{claim}\label{canonicalepsilon}
	$\forall \epsilon>0$ and $\epsilon'\in(0, \frac{2}{3}\epsilon]$, $\HWT_\chi^{\epsilon'}\subseteq \HWT_\chi^{\epsilon}$, and $\WT_\chi^{\epsilon'}\subseteq \WT_\chi^{\epsilon}$.
\end{claim}
For proof, see \cite[Claim~2.6]{LifeOfPi}.

\begin{lemma}\label{forChaptoro6ixo}
$\exists \epsilon_\chi>0$ which depends on $M,f,\beta$ and $\chi$ s.t $\forall \epsilon\in(0,\epsilon_\chi]$ $\HWT_\chi^\epsilon$ has a Markov partition and a finite-to-one coding with a H\"older continuous factor map, where the induced coding space is locally-compact.
\end{lemma}
See \cite{LifeOfPi} for full details.

\begin{definition}\label{NUHsharp}
 $$\HWT_\chi:=\left\{x\in\chi\text{-}\mathrm{summ}:x\text{ is recurrently }\epsilon_\chi\text{-weakly temperable}\right\}.$$ 
\end{definition}
\noindent\textbf{Remark:} $\HWT_\chi$ carries all $\chi$-hyperbolic $f$-invariant probability measures; and its definition depends only on the quality of hyperbolicity of the orbits of its elements. In the following parts of this paper, when $\chi>0$ is fixed, the subscript of $\epsilon_\chi$ would be omitted to ease notation.

\begin{definition}[Pesin set] 
	$\forall N\geq 1$, the set $$\Lambda_N:=\{x\in \WT_\chi^\epsilon:\exists \epsilon\text{-weakly tempered kernel of }x\text{ s.t }q(x)\geq \frac{1}{N}\},$$ is called the {\em Pesin set of level $N$} (also sometimes called {\em level set}, or {\em Pesin level set}).
\end{definition}
\begin{definition}[Positively recurrent points] $\forall N \geq 1$, denote by $\Lambda_N$ the Pesin set of level $N$. Let
	$$\HWT_\chi^\mathrm{PR}:=\left\{x\in \HWT_\chi:\exists N\in\mathbb{N}\text{ s.t }\limsup_{n\rightarrow \infty}\frac{1}{n}\sum_{k=0}^{n-1}\mathbb{1}_{\Lambda_N}(f^k(x))>0\right\}.$$
\end{definition}

For the next definition, we assume w.l.o.g that $\exists s\in\{1,...,d-1\}$ s.t for all points $x\in \HWT_\chi$, $\mathrm{dim}H^s(x)=s$. This is possible, since we may simply split $\HWT_\chi$ into $d-1$ disjoint sets, and all results apply to each subset.  
\begin{definition}[Grassmann-H\"older continuity]\label{MnfldWeakHolder}
Let $TM$ be the tangent bundle of $M$, and let $x,y\in \WT_\chi^\epsilon$. Let $d_s(H^s(x),H^s(y))$ be the Grassmannian distance (of subspaces of dimension $s$, over $TM$) between the stable space of $x$ and stable space of $y$. Similarly, let $d_u(H^u(x),H^u(y))$ be the Grassmanian (of subspaces of dimension $d-s$, over $TM$) distance between the unstable space of $x$ and stable space of $y$. The {\em Lyapunov distance} on $\WT_\chi^\epsilon$ is $d_{GH}(x,y):=d(x,y)+ d_s(H^s(x),H^s(y))+ d_u(H^u(x),H^u(y))$. A function $\varphi:\WT_\chi^\epsilon\rightarrow\mathbb{R}$ is called {\em Grassmann-H\"older continuous} if it is H\"older continuous w.r.t  $d_{GH}(\cdot,\cdot)$.

Let $E\supseteq \WT_\chi^\epsilon$ be a measurable subset of $\bigcup\{W^u(x):x\in\WT_\chi^\epsilon\}$. A function $\varphi:E\rightarrow\mathbb{R}$ is called {\em unstable-Grassmann-H\"older continuous on $E$} if it is H\"older continuous w.r.t  $d(\cdot,\cdot)+d_u(\cdot,\cdot)$ on $E$, where for $x\in \bigcup\{W^u(x):x\in\WT_\chi^\epsilon\}$ $H^u(x)$ is defined as $T_xW^u(x)$ (which is well-defined).
\end{definition}

Since $f\in\mathrm{Diff}^{1+\beta}(M)$, $\forall t>0$, the scaled geometric potential $\varphi_t(x)=-t\cdot\log\Jac(d_xf|_{H^u(x)})$ is unstable-$\cont$. Every H\"older continuous potential is unstable-$\cont$, and every unstable-$\cont$ potential on $\WT_\chi^\epsilon$, is $\cont$. 

Notice also that a $\cont$ potential must be bounded, since $M$ has a finite diameter
.

\begin{definition}[Ergodic homoclinic class]\label{homoclinicclass}The {\em ergodic homoclinic class} of a hyperbolic periodic point $p$ (w.r.t $\chi$) is
    $$H_\chi(p):=\left\{x\in \HWT\chi:W^u(x)\pitchfork W^s(o(p))\neq\varnothing,W^s(x)\pitchfork  W^u(o(p))\neq\varnothing\right\}.$$
    Here $o(p)=\{f^k(p)\}$, $\pitchfork$ denotes transverse intersections of full codimension, and $W^{s/u}(\cdot)$ are the global stable and unstable manifolds of the point, respectively.
    \end{definition}
 See Definition \ref{globalstableunstable} for the definition of global stable and unstable leaves. Definition \ref{homoclinicclass} was introduced in \cite{RodriguezHertz}, with a set of Lyapunov regular points replacing $\HWT_\chi$. By a well-known argument, using the Inclination Lemma (see \cite[Theorem~5.7.2]{BrinStuck}), every two ergodic homoclinic classes are either disjoint modulo all conservative measures,\footnote{Conservative measures are measures which satisfy Poincar\'e's recurrence theorem: for any set of positive measure, almost every point in it returns to it infinitely often.} or coincide. We caution the reader that $H_\chi(p)$ is not closed, and that this definition is different from the definition in \cite{NewhousePeriodicEquivalenceRelation}.

Every ergodic $\chi$-hyperbolic probability is carried by a unique ergodic homoclinic class (see Claim \ref{tobewritten}).

\begin{definition}[Local equilibrium state]
	Let $\varphi:\WT_\chi^\epsilon\rightarrow\mathbb{R}$ be an $\cont$ potential. Let $\mu$ be a $\chi$-hyperbolic $f$-invariant ergodic probability measure. Let $H_\chi(p)$ be the unique ergodic homoclinic class which carries $\mu$. $\mu$ is called a hyperbolic {\em local equilibrium state of $\varphi$ on $H_\chi(p)$}, if
		$$h_\mu(f)+\int \varphi d\mu=\sup\left\{h_\nu(f)+\int \varphi d\nu:\nu\text{ is an }f\text{-invariant probability measure on  }H_\chi(p)\right\}.$$ 
\end{definition}
Notice, $\int\varphi d\mu<\infty$ since $\varphi$ is bounded since it is $\cont$.

\begin{definition}[Local pressure]
	Let $p$ be a periodic $\chi$-hyperbolic point, and let $H_\chi(p)$ be its ergodic homoclinic class. The {\em local pressure of $\varphi$ on $H_\chi(p)$} is $$P_{H_\chi(p)} (\varphi):= \sup\left\{h_\nu(f)+\int \varphi d\nu:\nu\text{ is an }f\text{-invariant probability measures on  }H_\chi(p)\right\}.$$
\end{definition}
The boundedness of $\varphi$ implies $P_{H_\chi(p)}(\varphi)<\infty$. Let $E$ as in Definition \ref{MnfldWeakHolder}.
\begin{definition}
	Given a sequence of $C^1$ embedded submanifolds, $V^u_n$, we write {\em $V^u_n\xrightarrow[n\rightarrow\infty]{C^1} V^u$}, where $V^u$ is an embedded submanifold, if $\sup_{x\in V^u_n}\inf_{y\in V^u}\{d(x,y)+d_{TM}(T_xV^u_n,T_yV^u)\}\xrightarrow[n\rightarrow\infty]{}0$.
\end{definition}
\begin{definition}[$\varphi$-conformal system of measures]\label{Cara}
	Let $\varphi:E\rightarrow\mathbb{R}$ be an unstable-$\cont$ potential. Let $H_\chi(p)$ be the ergodic homoclinic class of a $\chi$-hyperbolic and periodic point $p$. A {\em $\varphi$-conformal system of measures} on $H_\chi(p)$ is a collection of non-zero Radon measures on local\footnote{``local" here means diffeomorphic to a ball, in the induced metric.} unstable leaves of points in $H_\chi(p)$ 
	such that
	\begin{enumerate}
	\item $V^u_n\xrightarrow[n\rightarrow\infty]{C^1}V^u$ implies $m_{V^u_n} ^\varphi(1) \xrightarrow[n\rightarrow\infty]{} m_{V^u} ^\varphi(1) $,
	 \item for every local unstable leaf of a point in $H_\chi(p)$, $V^u$, $$m^\varphi_{V^u}\circ f^{-1}=e^{\varphi-P_{H_\chi(p)}(\varphi)}\cdot m^\varphi_{f[V^u]},$$
	where $m^\varphi_{ V^u}, m^\varphi_{ f[V^u]} $ are the $\varphi$-conformal measures on the local unstable leaves $V^u, f[V^u]$, respectively.
\end{enumerate}
\end{definition}
An example of a conformal family of measures, is the Riemannian volume on unstable leaves. It is $
\varphi$-conformal with $\varphi$ being the geometric potential $-\log \Jac(df|_{\text{unstable leaves}})$, when the pressure is $0$ (e.g when an SRB exists).

It is not clear a-priori when do conformal families exist. 
 In \cite{CPZ}, Climenhaga, Pesin, and Zelerowicz give a construction of such families in the partially hyperbolic setup using the Carath\'eodory dimension structure (which extends their result in the uniformly hyperbolic setup \cite{CPZuniformlyHyp}), under the assumption that $\varphi$ has the Bowen property (i.e uniformly summable variations on stable and unstable leaves, under a certain regularity assumption). 

\begin{definition}[$\varphi$-invariant family of leaf measures] \label{pcwis}
		Let $\varphi:\WT_\chi^\epsilon\rightarrow\mathbb{R}$ be a $\cont$ potential. Let $H_\chi(p)$ be the ergodic homoclinic class of a periodic $\chi$-hyperbolic point $p$ s.t $P_{H_\chi(p)}(\varphi)<\infty$. A {\em $\varphi$-invariant family of leaf measures on $H_\chi(p)$}, is a family $\mathcal{F}_{H_\chi(p)}(\varphi) $ of non-zero and finite measures s.t
	\begin{enumerate}
	\item $\forall \mu\in\mathcal{F}_{H_\chi(p)}(\varphi) $,	$\mu$ is carried by $WT_\chi^\epsilon$, and its support is contained in a local unstable leaf of a point $x\in H_\chi(p)$,
	\item $\exists$ a map $A:\cup\{\mathrm{supp}(\mu')\}_{\mu'\in\mathcal{F}_{H_\chi(p)}(\varphi)}\rightarrow\mathbb{R}$ which is measurable when restricted to measurable sets, s.t $\forall \mu\in\mathcal{F} _{H_\chi(p)}(\varphi) $, 

$\mu\circ f^{-1}$ 
$=e^{-P_{H_\chi(p)} (\varphi)}\sum_{i=1}^{N_{\mu}}e^{\phi(\mu^{(i)})}\mu^{(i)}$, where $\mu^{(i)}\in\mathcal{F} _{H_\chi(p)} (\varphi)$ for all $1\leq i\leq N(\mu)<\infty$, and $\phi(\mu^{(i)}) $ is a constant s.t $\phi(\mu^{(i)})=\varphi(x)+A(x)-A(f^{-1}(x))$ for $\mu^{(i)}$-a.e $x\in M$ where $\varphi$ is defined.
	\item $\exists$ $\sigma$-compact metric space $X$ s.t $\mathcal{F} _{H_\chi(p)}(\varphi)=\{\mu_{x}\}_{x\in X}$ and $x_n\rightarrow y\Rightarrow \mu_{x_n}\xrightarrow[]{\text{weak-}*} \mu_{y}$
.
	\end{enumerate}
\end{definition}
The measures in $\mathcal{F}_{H_{\chi}(p)}(\varphi)$ are not required to be mutually singular. 

One cannot expect uniqueness of such a family, as the Radon-Nikodym derivative of the push-forwards can always be changed by a coboundary.
\subsection{Main Results}\label{mainresult}

\begin{theorem*}[A]
Let $\varphi:\WT_\chi^\epsilon\rightarrow\mathbb{R}$ be a 
$\cont$ potential. Let $H_\chi(p)$ be the ergodic homoclinic class of a periodic $\chi$-hyperbolic point $p$
. Then $H_\chi(p)$ admits a $\varphi$-invariant family of leaf measures, $\mathcal{F}_{H_\chi(p)}(\varphi)$.
\end{theorem*}
While conformal families of measures are generally a simpler object for calculations, their general existence is a hard problem which is only solved in a few setups (we mention a few in \textsection \ref{introintro}). The main advantage of invariant families, is that they exists with (almost) no preceding assumptions. Furthermore, when a conformal family does exist, and so does an equilibrium state for the respective potential, the invariant family measures must coincide in measure-class with the conformal family. See the last result of this section.

\begin{definition}[$\varphi$-leaf condition]\label{leafco}
Let $\varphi:\WT_\chi^\epsilon\rightarrow\mathbb{R}$ be a 
$\cont$ potential. Let $H_\chi(p)$ be an ergodic homoclinic class of a $\chi$-hyperbolic periodic point $p$
. We say that {\em the $\varphi$-leaf condition is satisfied for $H_\chi(p)$}, if $\exists \mu\in\mathcal{F}_{H_\chi(p)}(\varphi)$ s.t $\mu(\HWT_\chi^{\mathrm{PR}})>0$.
\end{definition}

Here $\varphi_n(q):=\sum_{k=0}^{n-1}\varphi(f^{-k}(q))$, and $\mathcal{F}_{H_\chi(p)}(\varphi)$ is given by the theorem above. That is, in the theorem above we construct an invariant family, and the leaf condition applies to any invariant family as constructed by that theorem; We remind the reader that the invariant family need not be unique, as the leaf measures can always be changed by a coboundary factor. 

\begin{theorem*}[B] 
Let $\varphi:\WT_\chi^\epsilon\rightarrow\mathbb{R}$ be a 
$\cont$ potential. Let $H_\chi(p)$ be an ergodic homoclinic class of a $\chi$-hyperbolic periodic point $p$ s.t $p$ belongs to a Pesin level set $ \Lambda_{l
}$
. Then $\varphi$ admits a $\chi$-hyperbolic local equilibrium state on $H_\chi(p)$, if and only if $\sum_{n\geq1}\sum_{f^n(q)=q,q\in H_\chi(p)\cap \Lambda_l}e^{\varphi_n(q)-n\cdot P_{H_\chi(p)}(\varphi)}=\infty$
and the $\varphi$-leaf condition is satisfied for $H_\chi(p)$. If the local equilibrium state exists, then it is unique and the $\varphi$-invariant family is its conditional measures.
\end{theorem*}
Note that if the assumption above holds for one $l$, then it holds for all $l'\geq l$, since $\Lambda_{l'}\supseteq \Lambda_l$.
\begin{cor*}[C]
	Local equilibrium states of $\cont$ potentials have local product structure.
\end{cor*}
 This follows from the absolute continuity property of the $\varphi$-invariant family we construct. This extends previous results for the Axiom A setup (see \cite{Haydn94,Leplaideur00}).
 
\begin{cor*}[D]
$(M,f)$ admits a $\chi$-hyperbolic equilibrium state for $\varphi$ if and only if there exists an ergodic homoclinic class $H_\chi(p)$ which satisfies the assumptions of the previous theorem, and in addition $P_{H_\chi(p)}(\varphi)=P_{\mathrm{top}}(\varphi)$ ($= \sup\{h_\nu(f)+\int \varphi d\nu:\nu\text{ is an }f\text{-inv. probablity on }M\}$).
\end{cor*}

\begin{theorem*}[E]
Let $\varphi:WT_\chi^\epsilon\rightarrow\mathbb{R}$ be a 
$\cont$ potential. Let $H_\chi(p)$ be an ergodic homoclinic class of a $\chi$-hyperbolic periodic point $p$ which admits a (unique) $\chi$-hyperbolic equilibrium state for $\varphi$
, $\nu_\varphi$. Assume that $H_\chi(p)$ admits a $\varphi$-conformal system of measures. Let $\mathcal{F}_{H_\chi(p)}(\varphi)$ be any $\varphi$-invariant family of conditional leaf measures as in Theorem $\mathrm{(A)}$, then there is a sub-family $\mathcal{F}'$ which disintegrates $\nu_\varphi$ (as in Theorem $\mathrm{(B)}$), and for which the following holds:
For any $\mu\in\mathcal{F}
'$, given a local unstable leaf $V^u_{\mathrm{loc}}$ which carries $\mu$, and a corresponding $\varphi$-conformal measure on $V^u_{\mathrm{loc}}$, $m^\varphi_{ V^u_\mathrm{loc}}$, there exists a measurable set of full $\mu$-measure, $E_\mu$, s.t $\mu\sim m^\varphi_{ V^u_\mathrm{loc}}|_{E_\mu}$. In particular, $m^\varphi_{ V^u_\mathrm{loc}} (\HWT_\chi^{\mathrm{PR}})>0$.
\end{theorem*}
The last property holds in particular when $\varphi$ is unstable-$\cont$ and is defined on unstable leaves, and thus we a-priori have no reason to know that its conformal measures ``see" the hyperbolic points.

If $\varphi$ is the geometric potential then the $\varphi$-conformal measures exist and are the induced Riemannian volume of the unstable leaves, and this theorem proves that every $\chi$-hyperbolic equilibrium state of the geometric potential with pressure $0$, has absolutely continuous conditional measures (and thus is an SRB measure). This is a new proof to the result by Ledrappier and Young in \cite{LedrappierYoungI}, when restricting to hyperbolic measures; and extends it to additional potentials.

\noindent\textbf{Remark:} In several of the main results, we use the assumption $\sum\limits_{n\geq1}\sum\limits_{f^n(q)=q,q\in H_\chi(p)\cap \Lambda_l}e^{\varphi_n(q)-n\cdot P_{H_\chi(p)}(\varphi)}$ $=\infty$ for some hyperbolic periodic point $p$ which belongs to a level set $\Lambda_{l}$ (where $\varphi:\WT_\chi^\epsilon\rightarrow\mathbb{R}$ is a 
$\cont$ potential). This assumption is used in order to show that $\varphi$ is recurrent (see \textsection \ref{rec1}) when lifted to an irreducible component which codes $H_\chi(p)$ (see \textsection \ref{mycorona}). This assumption is always used together with the leaf condition (see Definition \ref{leafco}). We do not know if the first assumption can be omitted (i.e does the leaf condition imply recurrence of the lifted potential). We were not able to prove that using the methods of this paper. The difference between this paper and \cite{SRBleaf}, where the leaf condition is sufficient, is the overlapping property of leaf measures (i.e the Lebesgue measures of two unstable leaves which intersect, coincide on the intersection). Here, we use a coding to construct our $\varphi$-invariant family of leaf measures, and in order to gain a bound on the overlapping of measures, we must use the bound on the multiplicity of the coding map $\wpi$; which in turn is only bounded on the recurrent part of the symbolic space in our construction; hence the need for recurrence. There are other constructions of codings with better multiplicity bounds for the coding map (e.g \cite{BuzziCoding}), though those multiplicity improvements come on expense of other properties which are necessary for this paper.

\section{Preliminary Constructions}\label{symbolicdynamicspart}
\subsection{Symbolic Dynamics}
Sinai \cite{SinaiMP} and Adler and Weiss \cite{AW70} constructed Markov partitions and for Anosov diffeomorphims and linear toral automorphisms, respectively. In \cite{Sarig} Sarig constructed a Markov partition for non-uniformly hyperbolic surface diffeomorphisms. We extended his results to manifolds of any dimension greater or equal to 2 in \cite{SBO}. In \cite{LifeOfPi} we modify the codings from \cite{SBO,Sarig} and introduce the set $\HWT_\chi$. $\HWT_\chi$ is defined canonically, and still consists of all recurrently-codable points (see Proposition \ref{cafeashter}) in the modified symbolic dynamics. In the following section we review those results.

Since $M$ is compact, $\exists r=r(M)>0,\rho=\rho(M)>0$ s.t the exponential map $\exp_x: \{v\in T_xM:|v|\leq r\}\rightarrow B_\rho(x)=\{y\in M: d(x,y)<\rho\}$ is injective.
\begin{definition}[Pesin-charts]\label{PesinCharts} When $\epsilon\leq r$, the following is well defined since $C_\chi(\cdot)$ is a contraction (see \cite[Lemma~2.9]{SBO}):
\begin{enumerate}
	\item $\psi_x^\eta:=\exp_x\circ C_\chi(x):\{v\in T_xM:|v|\leq \eta\}\rightarrow B_\rho(x)$, $\eta\in (0,Q_\epsilon(x)]$ is called a {\em Pesin-chart}.
	\item A {\em double Pesin-chart} is an ordered couple $\psi_x^{p^s, p^u}:=(\psi_x^{p^s}, \psi_x^{p^u})$, where $\psi_x^{p^s}$ and $\psi_x^{p^u}$ are Pesin-charts. 
\end{enumerate} 
\end{definition}

\begin{theorem}\label{mainSBO}
	 $\forall \chi>0$ s.t there exists a $\chi$-hyperbolic periodic point $p$, $\exists$ a countable and locally-finite\footnote{Finite out-going and in-going degree at each vertex.} directed graph $\mathcal{G}= \left(\mathcal{V},\mathcal{E}\right)$ which induces a topological Markov shift $
\Sigma:=\{\ul{u}\in\mathcal{V}^\mathbb{Z}:
(u_i,u_{i+1})\in	\mathcal{E},\forall i\in\mathbb{Z}\}$ s.t $\Sigma$ admits a factor map $\pi:\Sigma\rightarrow M$ with the following properties:
	 \begin{enumerate}
	 	\item $\sigma:\Sigma\rightarrow\Sigma$ is defined by $(\sigma \ul{u})_i:=u_{i+1}$, $i\in \mathbb{Z}$ (the left-shift); and satifies $\pi\circ\sigma=f\circ\pi$.
	 	\item $\pi$ is a H\"older continuous map w.r.t the metric $d(\ul{u},\ul{v}):=\exp\left(-\min\{i\geq0:u_i\neq v_i\text{ or } u_{-i}\neq v_{-i}\}\right)$.
	 	\item Let $\Sigma^\#:=\left\{\ul{u}\in \Sigma:\exists n_k,m_k\uparrow\infty\text{ s.t }u_{n_k}=u_{n_0}, u_{-m_k}=u_{-m_0},\forall k\geq0\right\}$, 
	 then $\pi[\Sigma^\#]$ carries all $f$-invariant, $\chi$-hyperbolic probability measures.
	 \end{enumerate}
\end{theorem}
This theorem is the content of \cite[Theorem~3.13]{SBO} (and similarly, the content of \cite[Theorem~4.16]{Sarig} when $d=2$). $\mathcal{V}$ is a collection of double Pesin-charts (see 
Definition \ref{PesinCharts}), 
which is discrete
.

\begin{definition}\label{Doomsday}
\text{ }
 
\medskip
\begin{enumerate}
	\item $\forall u\in \mathcal{V}$, $Z(u):=\pi[\{\underline{u}\in\Sigma^\#:u_0=u\}]$, $\mathcal{Z}:=\{Z(u):u\in\mathcal{V}\}$.
	\item $\mathcal{R}$ is defined to be a countable partition of $\bigcup\limits_{v\in\mathcal{V}}Z(v)=\pi[\Sigma^\#]$ s.t 
	\begin{enumerate}
	\item $\mathcal{R}$ is a refinement of $\mathcal{Z}$: $\forall Z\in\mathcal{Z},R\in\mathcal{R}$, $R\cap Z\neq\varnothing\Rightarrow R\subseteq Z$.
	\item $\forall v\in\mathcal{V}$, $\#\{R\in\mathcal{R}:R\subseteq Z(v)\}<\infty$ (\cite[\textsection~11]{Sarig}).
	\item The rectangles property: $\forall R\in\mathcal{R}$,$\forall x,y\in R$ $\exists ! z:=[x,y]_R\in R$, s.t $\forall i\geq0, R(f^i(z))= R(f^i(y)), R(f^{-i}(z))= R(f^{-i}(x))$, where $R(t):=$the unique partition member of $\mathcal{R}$ which contains $t$, for $t\in\pi[\Sigma^\#]$.
	\end{enumerate}
	\item $\forall R,S\in\mathcal{R}$, we say $R\rightarrow S$ if $R\cap f^{-1}[S]\neq\varnothing$, and let $\widehat{\mathcal{E}}:=\{(R,S)\in\mathcal{R}^2\text{ s.t }R\to S\}$.
	\item $\Sig:=\{\ul{R}\in\mathcal{R}^{\mathbb{Z}}: R_i\rightarrow R_{i+1},\forall i\in\mathbb{Z}\}$.
\end{enumerate}	
\end{definition}
\noindent\textbf{Remark:} Given $\mathcal{Z}$, the definition of $\mathcal{R}$ is proper, since such a refining partition as $\mathcal{R}$ exists by the Bowen-Sinai refinement, see \cite[\textsection~11.1]{Sarig}. By property (2)(b), and since $\Sigma$ is locally-compact (see Theorem \ref{mainSBO}, local-finiteness of $\mathcal{G}$ implies local-compactness of $\Sigma$), $\Sig$ is also locally-compact.
\begin{definition}\label{sigmasharp}\text{ }

\medskip
\begin{enumerate}
    \item $\Sig^\#:=\{\ul{R}\in\Sig:\exists n_k,m_k\uparrow\infty\text{ s.t }R_{n_k}=R_{n_0},R_{-m_k}=R_{-m_0},\forall k\geq0\}$.
    \item Two partition members $R,S\in \mathcal{R}$ 
are said to be {\em affiliated} if $\exists u,v\in\mathcal{V}$ s.t $R\subseteq Z(u), S\subseteq Z(v)$ and $Z(u)\cap Z(v)\neq\varnothing$ (this terminology is due to O. Sarig, \cite[\textsection~12.3]{Sarig}).
\end{enumerate}
\end{definition}
\begin{claim}[Local finiteness of the cover $\mathcal{Z}$]\label{localfinito}
	$\forall Z\in\mathcal{Z}$,$\#\{Z'\in\mathcal{Z}:Z'\cap Z\neq\varnothing\}<\infty$.
\end{claim}
This claim is the content of \cite[Theorem~5.2]{SBO} (and similarly \cite[Theorem~10.2]{Sarig} when $d=2$).
\noindent\textbf{Remark:} 
By Claim \ref{localfinito} and Definition \ref{Doomsday}(2)(b), it follows that every partition member of $\mathcal{R}$ has only a finite number of partition members affiliated to it. 

\begin{definition}\label{N_R}
	Let $R\in\mathcal{R}$, $N(R):=\#\{S\in \mathcal{R}: S\text{ is affiliated to }R\}$.
\end{definition}
\begin{theorem}
	Given $\Sig$ from Definition \ref{Doomsday}, there exists a factor map $\widehat{\pi}:\Sig\rightarrow M$ s.t 
	\begin{enumerate}
	\item $\widehat{\pi}$ is H\"older continuous w.r.t the metric $d(\ul{R},\ul{S})=\exp\left(-\min\{i\geq0: R_i\neq S_i\text{ or }R_{-i}\neq S_{-i}\}\right)$.
	\item $f\circ\widehat{\pi}=\widehat{\pi}\circ\sigma$, where $\sigma$ denotes the left-shift on $\Sig$.
	\item $\widehat{\pi}|_{\Sig^\#}$ is finite-to-one.
	\item $\forall \ul{R}\in \Sig$, $\widehat{\pi}(\ul{R})\in \overline{R_0}$.
	\item $\widehat{\pi}[\Sig^\#]$ carries all $\chi$-hyperbolic invariant probability measures.
	\end{enumerate}
\end{theorem}
This theorem is the content of the main theorem of [BO18], Theorem 1.1 (and similarly the content of [Sar13, Theorem 1.3] when $d = 2$). Notice, the fibers of $\wpi$ being finite in $\Sig^\#$ does not imply being uniformly bounded.
\begin{prop}\label{cafeashter}
$\widehat{\pi}[\Sig^\#]=\pi[\Sigma^\#]=\bigcupdot\mathcal{R}=\HWT_\chi$.
\end{prop}
This is the content of \cite[Proposition~4.11, Corollary~4.12]{LifeOfPi}.

\subsection{Maximal Dimension Unstable Leaves}
\begin{definition}
An {\em unstable leaf} (of $f$) in $M$, $V^u$, is a $C^{1+\frac{\beta}{3}}$ embedded, open,  Riemannian submanifold of $M$, such that $\forall x,y\in V^u$, $\limsup\limits_{n\rightarrow\infty} \frac{1}{n}\log d(f^{-n}(x),f^{-n}(y))<0$.
Similarly, a {\em stable leaf} is an unstable leaf of $f^{-1}$.
\end{definition}
\begin{definition}\label{maximaldimensionleaves}
An unstable leaf is called {\em an unstable leaf of maximal dimension}, if it is not contained in any unstable leaf of a greater dimension.
\end{definition}
Notice that if $x\in \HWT_\chi$ belongs to an unstable leaf of maximal dimension $V^u$, then $\mathrm{dim}H^u(x)=\mathrm{dim}V^u$. This can be seen using the following claim.
\begin{claim}\label{chikfila}
$\forall \ul{u}\in \Sigma$, there exists a maximal dimension unstable leaf $V^u(\ul{u})$, which depends only on $(u_i)_{i\leq0}$, and a maximal dimension stable leaf $V^s(\ul{u})$, which depends only on $(u_{i})_{i\geq0}$, s.t $\{\pi(\ul{u})\}=V^u(\ul{u})\cap V^s(\ul{u})$, and $f[V^s(\ul{u})]\subset V^s(\sigma\ul{u})$ and $f^{-1}[V^u(\ul{u})]\subset V^u(\sigma^{-1}\ul{u})$.	
\end{claim}
This is the content of \cite[Proposition~3.12,Theorem~3.13, Proposition~4.4]{SBO} (and similarly \cite[Proposition~4.15,Theorem~4.16,Proposition~6.3]{Sarig} when $d=2$). By construction, $V^s(\ul{u}),V^u(\ul{u})$ are local, in the sense that they have finite (intrinsic) diameter.

\subsection{Ergodic Homoclinic Classes and Maximal Irreducible Components}\label{mycorona}
\begin{definition}\label{globalstableunstable}
	Let $x\in \HWT_\chi$, and let $\ul{u}\in\Sigma^\#$ s.t $
	\pi(\ul{u})=x$. The {\em global stable (resp. unstable)} manifold of $x$ is $W^s(x):=\bigcup_{n\geq0}f^{-n}[V^s(\sigma^n\ul{u})]$ { \em(resp. }$W^u(x):=\bigcup_{n\geq0}f^{n}[V^u(\sigma^{-n}\ul{u})]$ {\em )}.
\end{definition}
This definition is proper and is independent of the choice of $\ul{u}$, by the construction of $V^u(\underline{u})$. For more details see \cite[Definition~2.23, Definition~3.2]{SBO}. Recall the remark after Definition    \ref{temperable}.

Let $p$ be a periodic point in $\chi\text{-}\mathrm{summ}$, i.e hyperbolic periodic point. Since $p$ is periodic, $\|C^{-1}_\chi(\cdot)\|$ is bounded along the orbit of $p$, and therefore $p\in \HWT_\chi$. 

\begin{definition}
	Let $\Sigma'$ be a topological Markov shift. A {\em cylinder} $[R_0,...,R_n]$, $n\geq0$, is the set $\{\ul{R}'\in \Sigma': R_i'=R_i \forall i=0,...,n\}$. An element of a topological Markov shift is called a {\em chain}.
\end{definition}

\begin{definition}\label{irreducibility}
Consider the Markov partition $\mathcal{R}$ from Definition \ref{Doomsday}.
\begin{enumerate}
	\item For any $a,b\in\mathcal{R}$, $n\in\mathbb{N}$, we write $a\xrightarrow[]{n}b$ if there exists a non-empty cylinder $[W_1,...,W_{n}]$ s.t $W_1=a$ and $W_n=b$.   
    \item Define $\sim\subseteq\mathcal{R}\times\mathcal{R}$ by $R\sim S\iff \exists n_{RS},n_{SR}\in\mathbb{N}\text{ s.t } R\xrightarrow[]{n_{RS}}S,S\xrightarrow[]{n_{SR}}R$
    . The relation $\sim$ is transitive and symmetric. When restricted to $\{R\in \mathcal{R}:R\sim R\}$, it is also reflexive, and thus an equivalence relation. Denote the corresponding equivalence class of some representative $R\in\mathcal{R}$, $R\sim R$, by $\langle R\rangle$.
    \item A {\em maximal irreducible component} in $\Sig$, corresponding to $R\in\mathcal{R}$ s.t $R\sim R$, is $\{\ul{R}\in\Sig: \ul{R}\in\langle R\rangle^\mathbb{Z}\}$.
    \item If a topological Markov shift is a maximal irreducible component of itself, we call it {\em irredcible}.
\end{enumerate}
\end{definition}

If $\wt{\Sigma}$ is a maximal irreducible component of $\Sig$, then $\wh{\pi}[\wt{\Sigma}^\#]$ is a subset of an ergodic homoclinic class, where $\wt{\Sigma}^\#:= \Sig^\#\cap\wt{\Sigma}$.

Recall Definition \ref{homoclinicclass}.

\begin{claim}\label{tobewritten}
	Let $\mu$ be an $f$-invariant ergodic probability measure which is carried by $\HWT_\chi$. Then there exists a unique ergodic homoclinic class $H_\chi(p)$, where $p$ is a $\chi$-hyperbolic periodic point, which carries $\mu$.
\end{claim}
For proof, see \cite{RodriguezHertz}.

\begin{prop}\label{homoclinicirreducible}
Let $p$ be a periodic $\chi$-hyperbolic point. Then, there exists a maximal irreducible component, $\widetilde{\Sigma}\subseteq\Sig$, s.t $\widehat{\pi}[\widetilde{\Sigma}^\#]\supseteq  H_\chi(p)\cap \wpi[\Sig^{\#\#}]$, where 
$$\Sig^{\#\#}=\{\ul{R}\in\Sig:\exists R\text{ s.t }\#\{i\geq0:R_i=R\}= \#\{i\leq0:R_i=R\}=\infty\}.$$ In particular, $\widehat{\pi}[\widetilde{\Sigma}^\#]=H_\chi(p)$ modulo all conservative measures.
\end{prop}

This is the content of \cite[Theorem~5.9]{LifeOfPi}, and is based on the result of Buzzi, Crovisier and Sarig in \cite{BCS} for homoclinic classes of the type of Newhouse \cite{NewhousePeriodicEquivalenceRelation}, and the $f$-invariant, $\chi$-hyperbolic, probability measures which they carry.

\medskip

\subsection{The Canonical Part of The Symbolic Space}\text{ }
\begin{definition}\label{littlerightshift}
\begin{align*}\Sig_L:=&\{(R_i)_{i\leq0}\in \mathcal{R}^{-\mathbb{N}}:\forall i\leq 0, R_{i-1}\rightarrow R_i\},\text{ }\sigma_R:\Sig_L\rightarrow\Sig_L,\sigma_R((R_i)_{i\leq0})=(R_{i-1})_{i\leq0}.
\end{align*}
\end{definition}
Notice, $\sigma_R$ is the right-shift, not the left-shift. In order to prevent any confusion, we will always notate $\sigma_R$ with a subscript $R$ (for ``right"), when considering the right-shift.

\begin{definition}[{\em The canonical coding} $\ul{R}(\cdot)$]\label{canonico}
Recall that $R(x)$ is the unique $\mathcal{R}$-element which contains $x$. Given $$x\in \pi[\Sigma^\#]=\bigcupdot\mathcal{R}\text{, let }(\ul{R}(x))_i:=R(f^i(x)), i\in\mathbb{Z}.$$
\end{definition}
One should notice that $\widehat{\pi}(\ul{R}(x))=x,\ul{R}(x)\in \Sig^\circ$ (see Definition \ref{canonicparts} for $\Sig^\circ$).

\begin{definition}\label{chikfilb}\label{unstabledisconnected} $$\forall \ul{R}\in\Sig_L,\text{ }W^u(\ul{R}):=\bigcap_{j=0}^\infty f^j[R_{-j}].$$
\end{definition}

$W^u(\ul{R})$ might be empty, and is not expected to be an immersed submanifold. It is a subset of $\HWT_\chi$. We have the following important property.
\begin{cor}\label{decomposition} $\forall \ul{R}\in\Sig_L$, $$f[W^u(\ul{R})]=\bigcupdot_{\sigma_R\ul{S}=\ul{R}}W^u(\ul{S}).$$
\end{cor}
See \cite[Corollary~3.19]{SRBleaf} for proof.

\begin{definition}\label{canonicparts}
$$\Sig^\circ:=\{\ul{R}\in\Sig:\forall n\in\mathbb{Z}, f^n(\widehat{\pi}(\ul{R}))\in R_n\},$$
$$\Sig^\circ_L:=\{\ul{R}\in\Sig^\#_L:W^u(\ul{R})\neq\varnothing\},$$
where $\Sig^\#_L:=\{(R_i)_{i\leq 0}:(R_i)_{i\in\mathbb{Z}}\in \Sig^\#\}$.
We call $\Sig^\circ,\Sig^\circ_L$ {\em the canonical parts of the respective symbolic spaces}.
\end{definition}
Notice that $\ul{R}(\cdot)$ is the inverse of $\widehat{\pi}|_{\Sig^\circ}$; and that $\Sig^\circ_L=(\Sig^\circ)_L:=\{(R_{i})_{i\leq0}:(R_i)_{\in\mathbb{Z}}\in\Sig^\circ\}$.

\begin{claim}\label{imagecanonic}
$$\widehat{\pi}[\Sig^\circ]=\widehat{\pi}[\Sig^\#]=\pi[\Sigma^\#]=\bigcupdot\mathcal{R}.$$
\end{claim}
See \cite[Corollary~3.21]{SRBleaf} for proof.

\begin{lemma}\label{nextstepcover}
If $R_0\rightarrow R_1$, where $R_0,R_1\in\mathcal{R}$, and $R_0\subseteq Z(u_0)$, $u_0\in\mathcal{V}$ then $\exists u_1\in\mathcal{V}$ s.t $u_0\rightarrow u_1$ and $R_1\subseteq Z(u_1)$. Analogously, if $R_0\rightarrow R_1$ and $R_1\subseteq Z(v_1)$, then $\exists v_0$ s.t $v_0\rightarrow v_1$ and $R_0\subseteq Z(v_0)$.
\end{lemma}
See \cite[Lemma~4.1]{SRBleaf} for proof.
\begin{definition}
Given a chain $\underline{R}\in \Sig_L$ we say that a chain $\underline{u}\in\mathcal{V}^{-\mathbb{N}}$ {\em covers} the chain $\underline{R}$ if $\underline{u}$ is admissible and $R_i\subseteq Z(u_i)$ for all $i\leq0$. We write $\ul{u}\CAR \ul{R}$.
\end{definition}
By Lemma \ref{nextstepcover}, for every chain $\underline{R}\in \Sig_L$, the collection of chains which cover $\underline{R}$ is not empty.

\begin{definition}
	$\Sigma_L:=\{(u_i)_{i\leq0}:\ul{u}\in\Sigma\}$.
\end{definition}

\begin{prop}
$\forall \ul{R}\in\Sig_L$, $\exists$ a local unstable leaf of maximal dimension $V^u(\ul{R})$ s.t $f^{-1}[V^u(\ul{R})]\subseteq V^u(\sigma_R\ul{R})$, and $V^u(\ul{R})$ is an open submanifold of $M$, with finite and positive induced Riemannian volume. In addition, $V^u(\ul{R})\subseteq \psi_{x}[\{v:|v|_\infty\leq p^u\}]$ for all $u=\psi_{x}^{p^s,p^u}\in \mathcal{V}$ s.t $Z(u)\supseteq R_0$, and $\ul{R}\mapsto V^u(\ul{R})$ is continuous in $C^1$-norm in any local choice of coordinates. Moreover, $V^u(\ul{R})\subseteq V^u(\ul{u})$ $\forall \ul{u}\in \Sigma_L$ s.t $\ul{u}\CAR\ul{R}$.
\end{prop}
See \cite[\textsection~4.1]{SRBleaf} for proof.

\section{A Markovian and Absolutely Continuous Family of Measures on the Symbolic Space}

\subsection{The Ruelle Operator and Sinai's Theorem}
\begin{definition}[Ruelle operator]\label{Ruelleo}
Let $\Tigma_L$ be a one-sided irreducible locally compact topological Markov shift (of negative chains), and let $\phi:\Tigma_L\rightarrow\mathbb{R}$ be a H\"older continuous potential. The associated {\em Ruelle operator} $L_\phi:C(\Tigma_L)\rightarrow C(\Tigma_L)$, is defined by $$(L_\phi h)(\ul{R}):=\sum_{\tigma\ul{S}=\ul{R}}e^{\phi(\ul{S})}h(\ul{S}),$$
where $h\in C(\Tigma_L)$, and $\tigma:\Tigma_L\rightarrow\Tigma_L$ is the right-shift.
\end{definition}

\begin{definition}
	Let $X$ be a topological Markov shift (either one-sided or two-sided). A function $\wh{\varphi}:X\rightarrow \mathbb{R}$ is called {\em weakly H\"older continuous}, if $\exists C>0,\alpha>0$, s.t $\forall \ul{x},\ul{y}\in X$ s.t $d(\ul{x},\ul{y})\leq e^{-1}$, we have $|\wh{\varphi}(\ul{x})-\wh{\varphi}(\ul{y})|\leq C\cdot d(\ul{x},\ul{y})^\alpha$.
\end{definition}
This notion is introduced in \cite{SarigTDF}, with the weaker assumption of $d(\ul{x},\ul{y})\leq e^{-2}$. 

Recall Definition \ref{MnfldWeakHolder}.
\begin{claim}\label{forweakholder}
Let $\varphi:WT_\chi^\epsilon\rightarrow\mathbb{R}$ be a potential which is $\cont$. Let $\wh{\varphi}:\Sig\rightarrow \mathbb{R}$, $\wh{\varphi}:=\varphi\circ\wpi$, be the lift of $\varphi$. Then $\wh{\varphi}$ is weakly H\"older continuous on $\Sig$.
\end{claim}
This follows from \cite[Proposition~6.1]{SBO}.

\begin{theorem}[Sinai]\label{SinaiBowen}
Let $\varphi: \HWT_\chi\rightarrow\mathbb{R}$ be a $\cont$ potential. Then there exist two functions $\varphi^*,A: \bigcupdot \mathcal{R}\rightarrow\mathbb{R}$ such that $\varphi^*(x)= \varphi^*((R(f^{-i}(x)))_{i\geq0})$ depends only on the past of the itinerary of $x$, $\varphi^*=\varphi+A-A\circ f^{-1}$, $A$ is bounded, and $\wh{A}:=A\circ \wpi, \phi:=\varphi^*\circ \wpi$ are weakly H\"older continuous.
\end{theorem}
This theorem appears in Sinai's fundamental paper \cite{SinaiGibbs}. He proves it in the context of a topological Markov shift with a finite alphabet (for the countable case see \cite{Daon}). We give the proof in the case when a countable Markov partition exists.
\begin{proof}
	First, $\forall a\in \mathcal{R}$, fix a point $y_a\in a$. $\forall x\in\bigcupdot\mathcal{R}$, define $x^*:=[x,y_{R(x)}]_{R(x)}\in R(x)$, where $R(x)$ is unique partition element of $\mathcal{R}$ which contains $x$, and $[\cdot,\cdot]_{R(x)}$ denotes the Smale bracket of two points in $R(x)$ (the right-hand term yields its future, and the left-hand terms yields its past). Define $A:\bigcupdot\mathcal{R}\rightarrow\mathbb{R}$,
	$$A(x):=\sum_{n\geq0} \varphi(f^{-n}(x^*))-\varphi(f^{-n}(x)).$$
By Claim \ref{forweakholder}, $\wh{\varphi}:=\varphi\circ\wpi$ is weakly H\"older continuous. In addition, $d(\ul{R}(f^{-n}(x)), \ul{R}(f^{-n}(x^*)))\leq e^{-n}$. Let $C>0,\theta\in (0,1)$ be the weak H\"older constants of $\wh{\varphi}$ s.t $d(\ul{R}^{(1)}, \ul{R}^{(2)})\leq e^{-n}\Rightarrow$ $|\wh{\varphi}(\ul{R}^{(1)})-\wh{\varphi}(\ul{R}^{(2)})|\leq C\cdot \theta^n$. Then $\wh{A}$ is a uniformly convergent series with uniformly bounded and equicontinuous summands, and so $\wh{A}$ is bounded and continuous
. In particular, $A$ is also bounded.

Let $\varphi^*:=\varphi+A-A\circ f^{-1}$ and $\phi:=\varphi^*\circ\wpi$. The weak H\"older continuity of $\wh{A}=A\circ \wpi$ and of $\phi$ follows from the weak H\"older continuity of $\wh{\varphi}$ and the proof of the finite alphabet case in \cite{SinaiGibbs}.	

It remains to show that $\varphi(x)= \varphi((R(f^{-i}(x)))_{i\geq0})$.
\begin{align*}
\varphi^*(x)&=\varphi(x)+A(x)-A(f^{-1}(x))=\\
&=\varphi(x)+ \sum_{n\geq0} \varphi(f^{-n}(x^*))-\varphi(f^{-n}(x))-\sum_{n\geq0} \varphi(f^{-n}((f^{-1}(x))^*))-\varphi(f^{-n-1}(x))\\
&=\varphi(x^*)+ \sum_{n\geq0} \varphi\left(f^{-n}\left(f^{-1}(x^*)\right)\right)-\varphi\left(f^{-n}\left(\left(f^{-1}(x)\right)^*\right)\right).
\end{align*}
The last expression depends only on $(R(f^{-i}(x)))_{i\geq0}$.

\end{proof}
\noindent\textbf{Remark:} Although $\wh{\varphi},\wh{A},\phi$ are defined on $\Sig^\circ$, they extend continuously to $\Sig$ because of their uniform moduli of continuity and the fact that $\Sig^\circ$ is dense in $\Sig$. Thus $\phi:\Sig_L\rightarrow\mathbb{R}$ is well defined and weakly H\"older continuous.

\subsection{Recurrence, Harmonic Functions, and Conformal Measures}\label{rec1}
As before, $\Tigma_L$ is a maximal irreducible component of $\Sig_L$, and $\phi:\Tigma_L\rightarrow\mathbb{R}$ is a weakly H\"older continuous potential.

In this section we use the Gurevich pressure, and we begin by recalling its definition. For a weakly H\"older potential $\zeta:\Sig_L\rightarrow\mathbb{R}$ , the {\em local partition functions} are for $n\geq 1$, $$Z_n(\zeta,R):=\sum\limits_{\ul{S}\in \Sig_L\cap[R],\sigma_R^n\ul{S}=\ul{S}}e^{\sum_{k=0}^{n-1}\zeta(\sigma_R^k\ul{S})}, $$ where the sum has finitely many terms since $\Sig_L$ is locally compact. The {\em Gurevich pressure} of the potential $\zeta$ is $$P_G(\zeta):=\limsup\limits_{n\rightarrow\infty}\frac{1}{n}\log Z_n(\zeta,R)\in(-\infty,\infty].$$ When $\Sig_L$ is irreducible, the limit is independent of the choice of the symbol $R$ (see \cite[Proposition~3.2]{SarigTDF}). We say that the potential $\zeta$ is {\em recurrent} if for some symbol $R$, $$\sum_{n\geq1}e^{-nP_G(\zeta)}Z_n(\zeta,R)=\infty.$$ In this case the sum diverges for all $R$, see \cite[Corollary~3.1]{SarigTDF}. We say that the potential $\zeta$ is {\em positive recurrent} if it is recurrent, and for some symbol $R$, $$\sum_{n\geq1}n\cdot e^{-nP_G(\zeta)} \sum\limits_{\substack{\ul{S}\in \Sig_L\cap[R]\text{, s.t }\sigma_R^n\ul{S}=\ul{S},\\ \text{ and }\forall 0< l<n, S_{-l}\neq R}}e^{\sum_{k=0}^{n-1}\zeta(\sigma_R^k\ul{S})}<\infty.$$  Again, this property turns out to be independent of $R$. For more details, see \cite[\textsection~3.1.3]{SarigTDF}. 
For a detailed review of the properties of recurrent potentials, see \cite{SarigTDF}.

\begin{definition}
	Let $\phi:\Tigma_L\rightarrow\mathbb{R}$ be a weakly H\"older continuous potential s.t $P_G(\phi)<\infty$, where $\Tigma_L$ is a maximal irreducible component of $\Sig_L$. A positive and continuous function $\psi:\Tigma_L\rightarrow\mathbb{R}^+$ is called {\em $\phi$-harmonic} if $L_\phi\psi= e^{P_G(\phi)}\psi$. A non-zero Radon measure $p$ on $\Tigma_L$ is called {\em $\phi$-conformal} if $L_\phi^*p= e^{P_G(\phi)}p$, where $L_\phi^*$ is the dual operator of $L_\phi$.
\end{definition}
 
 Notice, $\Tigma_L$ may be non-compact, and thus $p$ may be infinite.

\begin{theorem}
	Let $\phi:\Tigma_L\rightarrow\mathbb{R}$ be a weakly H\"older continuous potential, where $\Tigma_L$ is an irreducible locally compact one-sided topological Markov shift. If $P_G(\phi)<\infty$ then there exist a $\phi$-harmonic function $\psi$, and a $\phi$-conformal measure $p$.
\end{theorem}
This theorem is due to work of Sarig \cite{SarigPR,SarigNR}, Cyr \cite{Cyr}, and Shwartz \cite{Shwartz}. Sarig had proven that when $\phi$ is recurrent, $\psi$ and $p$ exist and are unique up to scaling. In the transient case, Cyr showed the existence of a $\phi$-conformal measure, and Shwartz proved the existence of a $\phi$-harmonic function with a weakly H\"older logarithm. In the transient case, the $\phi$-harmonic functions and the $\phi$-conformal measures are not always unique up to scaling. For the structure of the set of these objects, see \cite{Shwartz}.
\begin{lemma}\label{logHarmonicHolder}
	Let $\phi:\Tigma_L\rightarrow\mathbb{R}$ be a weakly H\"older continuous potential, where $\Tigma_L$ is an irreducible locally compact one-sided topological Markov shift. Let $\psi:\Tigma_L \rightarrow\mathbb{R}^+$ be a $\phi$-harmonic function. Then $\log\psi$ is weakly H\"older continuous.
\end{lemma}
In the recurrent case, Sarig proves it in \cite[Proposition~3.4]{SarigTDF}. In the transient case, it follows from the proof of Theorem 5.7 in \cite{Shwartz}. Shwartz shows in the proof that $\log\psi$ is uniformly continuous. A close inspection of his argument tells that in fact, if $\phi$ is weakly H\"older continuous, then so is $\log\psi$.

\subsection{Recurrence and Periodic Orbits}

\begin{lemma}\label{periodicRecurrence}
Let $\varphi: \HWT_\chi\rightarrow\mathbb{R}$ be a 
$\cont$ potential, and let $\phi:\Sig_L\rightarrow \mathbb{R}$ be a corresponding weakly H\"older continuous one-sided potential given by Theorem \ref{SinaiBowen} and the remark following it. Let $H_\chi(p)$ be an ergodic homoclinic class of a periodic and $\chi$-hyperbolic point $p$, and let $\Tigma$ be a maximal irreducible component of $\Sig$ s.t $\wpi[\Tigma^\#]= H_\chi(p)$ modulo conservative measures (see Proposition \ref{homoclinicirreducible}). Then $\phi:\Tigma_L\rightarrow\mathbb{R}$ is recurrent if and only if 
$$\sum_{n\geq1}\sum_{\substack{f^n(q)=q,\\q\in H_\chi(p)\cap \Lambda_l}}e^{\varphi_n(q)-n\cdot P_{H_\chi(p)}(\varphi)}=\infty,$$
where $\Lambda_l$ is a level set s.t $p\in \Lambda_{l\cdot e^{-2\sqrt[3]{\epsilon}}}$ and $\varphi_n(q)=\sum_{k=0}^{n-1}\varphi(f^{-k}(q))$.
\end{lemma}
\begin{proof}
	We start with the ``if" part. Let $q\in H_\chi(p)\cap \Lambda_l$ be a periodic $\chi$-hyperbolic point of period $n$. By Proposition \ref{homoclinicirreducible}, $q$ has a (periodic\footnote{Let $\underline{R}\in \Sig^\#\cap\wpi^{-1}[\{q\}]$. Then for any $m\geq0$, $\wpi(\sigma^{m\cdot n}\underline{R})=q$, whence by the finiteness of $\Sig^\#\cap\wpi^{-1}[\{q\}]$, $\exists m_1,m_2
	\in \mathbb{N}$ s.t $\sigma^{m_1\cdot n}\underline{R}=\sigma^{m_2\cdot n}\underline{R}$, and so $\sigma^{m\cdot n}\underline{R}=\underline{R}$ with 
	$m=\max\{m_1,m_2\}-\min\{m_1,m_2\}$.}) coding in $\Tigma$. Let $\ul{S}^q$ be such a periodic coding
	.
	
By \cite[Claim~7.6]{SRBleaf}, there is a finite collection of symbols $\mathcal{R}_l\subseteq \mathcal{R}$, which depends only on $l$, and which contains $S_0^q$.

By \cite[Theorem~1.3]{SBO} there exists $C_l=\max\{N(R)^2:R\in \mathcal{R}_l\}\in \mathbb{N}$ s.t $|\wpi^{-1}[\{q\}]\cap \Tigma^\#|\leq C_l$ (see Definition \ref{N_R} for the definition of $N(R)$). In addition, by Theorem \ref{SinaiBowen}, $\exists C:=2\|A\|_\infty<\infty$ s.t 
\begin{equation}\label{refereeAsked}
\phi_n(\ul{S}^q)=\pm C+\varphi_n(q),\text{ } 
\forall n\in\mathbb{N}.
\end{equation} 

Since $\varphi$ is bounded (recall remark after Definition \ref{MnfldWeakHolder}), so is $\phi$. By the spectral decomposition (see \cite[Theorem~2.5]{SarigTDFSymposium}), we may assume w.l.o.g that $(\Tigma,\sigma)$ is topologically mixing (otherwise apply the proof to each mixing component of a suitable period of $\sigma_R$). By the variational principle (see \cite[Theorem~4.4]{SarigTDF}, and \cite{VPunbdd} when the potential is not bounded), the Gurevich pressure of $\phi $ on $\Tigma_L$, $P_G(\phi)$, satisfies $$P_G(\phi)=\sup\{h_{\nu}(\tigma)+\int_{\Tigma_L} \phi d\nu:\nu\text{ is an invariant probability measure on }\Tigma_L\}.$$ Every invariant probability measure on $\Tigma_L$, $\nu$, extends uniquely to an invariant probability measure on $\Tigma$, $\nu'$, s.t $\nu'\circ \tau^{-1}=\nu$ and $h_{\nu}(\tigma)= h_{\nu'}(\sigma)$, where $\tau:\Tigma\rightarrow\Tigma_L$ is the projection onto the non-positive coordinates. Thus, $$P_G(\phi)=\sup\{h_{\nu'}(\sigma)+\int_{\Tigma} \phi d\nu':\nu'\text{ is an invariant probability measure on }\Tigma\}.$$ In addition, every invariant probability measure on $H_\chi(p)$, $\mu$, lifts to a an invariant probability measure on $\Tigma$, $\mu'$, s.t $\mu'\circ\wpi^{-1}=\mu$ and $h_{\mu}(f)= h_{\mu'}(\sigma)$. Then, we get 

\begin{equation}\label{underpressure2013}
P_G(\phi)=P_{H_\chi(p)}(\varphi).	
\end{equation}
By the Ruelle inequality the topological entropy of $f$ is finite, and since $\varphi$ is bounded, $P_{H_\chi(p)}(\varphi)<\infty$.

For every $n\in\mathbb{N}$ $\exists a_n,b_n\in\mathcal{R}_l$ s.t

\begin{align*}
	\sum_{\substack{f^n(q)=q,q\in H_\chi(p)\cap \Lambda_l\\ \ul{S}^q\in\wpi^{-1}[\{q\}]\cap \Tigma^\# ,S^q_0=a_n,S^q_n={b_n}}}e^{\phi_n(\ul{S}^q)}\geq \frac{1}{|\mathcal{R}_l|^2}\cdot \sum_{\substack{f^n(q)=q,q\in H_\chi(p)\cap \Lambda_l \\ \ul{S}^q\in\wpi^{-1}[\{q\}]\cap \Tigma^\#}}e^{\phi_n(\ul{S}^q)}.
\end{align*}
 
For every $n\geq1$, $\forall q\in H_\chi(p)\cap \Lambda_l$ s.t $f^n(q)=q$,  $|\wpi^{-1}[\{q\}]\cap \Tigma^\#|
\leq C_l$. By assumption, $\infty=\sum\limits_{n\geq1}\sum\limits_{f^n(q)=q,q\in H_\chi(p)\cap \Lambda_l}e^{\varphi_n(q)-n\cdot P_{H_\chi(p)}(\varphi)}$. Recall \eqref{refereeAsked}, then by the pigeonhole principle $\exists a,b\in\mathcal{R}_l$ and $n_k\uparrow\infty$ s.t $a_{n_k}=a$ and $b_{n_k}=b$, and so that the following sum is infinite: 
\begin{align*}
	\infty=\sum_{k\geq0 }\sum_{\substack{f^{n_k}(q)=q,q\in H_\chi(p)\cap \Lambda_l \\\ul{S}^q\in\wpi^{-1}[\{q\}]\cap \Tigma^\# ,S^q_0=a,S^q_{n_k}=b
	}}e^{\phi_{n_k}(S^q)-n_k\cdot P_{G}(\phi)}.
\end{align*}

Let $C_H'>0,\gamma\in(0,1)$ be the H\"older constant and H\"older exponent 
 of $\phi$, respectively. Let $C_H:=C_H'\cdot\sum_{n\geq0}\gamma^n<\infty$.

Let $\ul{W}$ be an admissible word s.t $W_0=b$ and $W_{m-1}=a$,   where $m:=|\ul{W}|$. For every finite word $\ul{W}'$ s.t $W'_0=W'_{|\ul{W}'|-1}$, let $\ul{S}^{\ul{W}'} $ denote the periodic extension of $\ul{W}'$ to a bi-infinite chain (in $\Tigma^\#$). 

For every $n\geq0$, for every finite word $\ul{W}'$ s.t $|\ul{W}'|=n$, every extension of $\ul{W}'$ to $\Tigma^\#$ which codes a periodic point in $H_\chi(p)\cap\Lambda_l$ of period $n$, must code the same point, by the shadowing lemma (see \cite[Proposition~3.12(4)]{SBO}). Therefore 

\begin{align}\label{MKT}
	\infty=&
	\sum_{k\geq0}\sum_{\substack{f^{n_k}(q)=q,q\in H_\chi(p)\cap \Lambda_l \\ \ul{S}^q\in\wpi^{-1}[\{q\}]\cap \Tigma^\# ,S^q_0=a,S^q_{n_k}=b}}e^{\phi_{n_k}(\ul{S}^q)-n_k\cdot P_G(\phi)}\nonumber\\
	\leq&  C_l e^{C_H'\sum_{j=1}^{n_k}\gamma^j}
	 \sum_{k\geq0} \sum_{\overset{}{ \wt{\ul{W}}:|\wt{\ul{W}}|=n_k, \wt{W}_0=a,\wt{W}_{n_k-1}=b}}e^{\phi_{n_k}(\ul{S} ^{( \wt{\ul{W}}\cdot\ul{W})}) -n_k\cdot P_G(\phi)}\nonumber\\
	\leq& C_l
	e^{m\cdot(\|\varphi\|_\infty+P_G(\phi))+C_H}\cdot \sum_{k\geq0}\sum_{\ul{S}\in\Tigma: \sigma^{n_k+m}\ul{S}=\ul{S},S_0=a} 
	e^{\phi_{n_k+m}(\ul{S}) -(n_k+m)\cdot P_G(\phi)}\nonumber\\
	\leq& C_l
	e^{m\cdot(\|\varphi\|_\infty+P_G(\phi))+C_H}\cdot\sum_{n\geq 1} \sum_{\ul{S}\in\Tigma: \sigma^{n}\ul{S}=\ul{S},S_0=a} e^{\phi_{n}(\ul{S}) -n\cdot P_G(\phi)}.\nonumber
\end{align}
Therefore $\phi$ is recurrent on $\Tigma$. 

It remains to show the ``only if" part. Assume that $\phi$ is recurrent on $\Tigma$, then $\forall a\in \mathcal{R}_l$, $$\sum\limits_{n\geq 1} \sum\limits_{\overset{\ul{S}\in\Tigma:}{ \sigma^{n}\ul{S}=\ul{S},S_0=a}} e^{\phi_{n}(\ul{S}) -n\cdot P_G(\phi)}=\infty.$$  Choose $a\in\mathcal{R}_l$ s.t $[a]\cap \wpi^{-1}[\{p\}]\cap\Tigma^\#\neq\varnothing$. $\forall n\geq 1$, every periodic chain $\ul{S}\in \Tigma\cap [a]$ of period $n$, codes a periodic point of period $n$ in $H_\chi(p)\cap \Lambda_l$ (see \cite[Theorem~4.13]{SBO}). There can be no more than $C_l$ chains which code the same point, thus $$\infty=\sum_{n\geq 1} \sum_{\ul{S}\in\Tigma: \sigma^{n}\ul{S}=\ul{S},S_0=a} e^{\phi_{n}(\ul{S}) -n\cdot P_G(\phi)}\leq C_l e^Ce^{C_H}\sum_{n\geq 1} \sum_{f^n(q)=q,q\in H_\chi(p)\cap \Lambda_l} e^{\varphi_{n}(q) -n\cdot P_{H_\chi(p)}(\varphi)}.$$
\end{proof}

\section{Absolutely Continuous Invariant Family of Leaf Measures}
 The following theorem is a version of the Ledrappier theorem (\cite{Led74},\cite[Theorem~3.3]{LLS}). We give a new proof here which is good for our needs, while the previous proofs provide a formula for the measures of cylinders.
  
\begin{theorem}\label{gauraluilui}
Let $\phi:\Tigma_L\rightarrow\mathbb{R}$ be a weakly H\"older continuous potential s.t $P_G(\phi)<\infty$, where $\Tigma_L$ is an irreducible locally compact one-sided topological Markov shift.
Let $\psi$ be a $\phi$-harmonic function on $\Tigma_L$. Then, $\exists$ a family of probability measures on $\Tigma$, $\{\wh{p}_{\ul{R}}\}_{\ul{R}\in \Tigma_L}$, s.t $\forall \ul{R}\in \Tigma_L$, $\wh{p}_{\ul{R}}$ is carried by $
	\{\ul{v}\in \Tigma:v_i=R_i,\forall i\leq0\}$ and $\wh{p}_{\ul{R}}\circ \sigma^{-1}=\sum_{\tigma\ul{S}=\ul{R}}e^{\phi(\ul{S})+\log\psi(\ul{S})-\log\psi\circ\tigma(\ul{S})-P_G(\phi)} \wh{p}_{\ul{S}}$, where $\tigma:=\sigma_R|_{\Tigma_L}$.
\end{theorem}
\begin{proof}
Fix $[R]\subseteq\Tigma_L$, and let $\ul{R}\in 
\Tigma_L\cap[R]$. Let $\tau:\Sig\rightarrow \Sig_L$ be the projection to the non-positive coordinates. Write $\wt{\phi}:=\phi+\log\psi-\log\psi\circ\tigma-P_G(\phi)$, and $\wt{\phi}_n:=\sum_{k=0}^{n-1}\wt{\phi}\circ\tigma^k$. It follows that $L_{\wt{\phi}}1=1$. Define,
 \begin{equation}\label{greendoor2}\wh{p}_{\ul{R}}^{(n)}:=\sum_{\tigma^n\ul{S}=\ul{R}}e^{\wt{\phi}_n(\ul{S})}\delta_{\sigma^{-n}\ul{S}^\pm},
 \end{equation}
 where $\delta_{\ul{S}^{\pm}}$ is a Dirac measure with mass at the chain $\ul{S}^{\pm}$, and $\ul{S}^{\pm}$ is some (any) choice of a chain in $\tau^{-1}[\{\ul{S}\}]\cap \Tigma$. We will show that the limit $\wh{p}_{\ul{R}}^{(n)}\xrightarrow[n\rightarrow\infty]{\text{weak-}*}\wh{p}_{\ul{R}}$ exists, and is independent of the choice of the representatives $\ul{S}^\pm\in \tau^{-1}[\{\ul{S}\}]\cap\Tigma$.

Step 1: $\forall n\geq0$, $\wh{p}_{\ul{R}}^{(n)}(1)=(L_{\wt{\phi}}^n1)(\ul{R})= 
1$. $\{\widetilde{\ul{R}}\in 
\Tigma\cap[R]: \widetilde{R}_i=R_i\forall i\leq 0\}$ is a compact set which carries $\wh{p}_{\ul{R}}^{(n)}$, $\forall n\geq 1$. Thus, $\exists n_k\uparrow\infty$ and a probability $\wh{p}_{\ul{R}}$ on 
$\{\widetilde{\ul{R}}\in 
\Tigma\cap[R]: \widetilde{R}_i=R_i,\forall i\leq 0\}$ s.t $\wh{p}_{\ul{R}}^{(n_k)}\xrightarrow[k\rightarrow\infty]{}\wh{p}_{\ul{R}} $.

Step 2: Let $[\ul{w}]_a:=\sigma^{a}[w_a,w_{a+1},...,w_{b-1},w_b]$, where $a,b\in\mathbb{Z}$ and $a\leq b$
. Let $n,m\geq0$. Observe:
\begin{enumerate}
	\item If $b\leq0$, $\wh{p}_{\ul{R}}^{(n)}(\mathbb{1}_{[\ul{w}]_a})=1$ if $\ul{R}\in[\ul{w}]_a$, and $\wh{p}_{\ul{R}}^{(n)}(\mathbb{1}_{[\ul{w}]_a})=0$ otherwise.
	\item If $a\leq 0\leq b$, $\wh{p}_{\ul{R}}^{(n)}(\mathbb{1}_{[\ul{w}]_a})= \mathbb{1}_{\sigma^a[w_a,\ldots,w_0]}(\underline{R})\cdot \wh{p}_{\ul{R}}^{(n)}(\mathbb{1}_{[w_0,...,w_b]})$ (and in particular, this equals $0$ if $w_0\neq R$).
	\item Assume $\ul{w}=R,w_1,...,w_b$, then $[\underline{w}]_a=[\underline{w}]$. If $n>b$,
	\begin{align}\label{greendoor}
		\wh{p}_{\ul{R}}^{(n+m)}(\mathbb{1}_{[\ul{w}]})=&\sum_{\tigma^m\ul{S}=\ul{R}}e^{\wt{\phi}_m(\ul{S})}\left(\sum_{\tigma^n\ul{Q}=\ul{S}}e^{\wt{\phi}_n(\ul{Q})}\delta_{\sigma^{-n-m}\ul{Q}^\pm}(\mathbb{1}_{[\ul{w}]})\right)\nonumber\\
		=& \sum_{\tigma^m\ul{S}=\ul{R}}e^{\wt{\phi}_m(\ul{S})}\left(\sum_{\tigma^n\ul{Q}=\ul{S}}e^{\wt{\phi}_n(\ul{Q})}\delta_{\sigma^{-n}\ul{Q}^\pm}(\mathbb{1}_{[\ul{w}]}\circ \sigma^{-m})\right)\nonumber\\
		=&\sum_{\tigma^m\ul{S}=\ul{R}}e^{\wt{\phi}_m(\ul{S})}\wh{p}_{\ul{S}}^{(n)}(\mathbb{1}_{[\ul{w}]}\circ \sigma^{-m})= \sum_{\tigma^m\ul{S}=\ul{R}}e^{\wt{\phi}_m(\ul{S})}\left(\wh{p}_{\ul{S}}^{(n)}\circ \sigma^m\right)(\mathbb{1}_{[\ul{w}]}).
	\end{align}
	Notice that \eqref{greendoor} holds regardless of the choice of $\ul{S}^\pm,\ul{Q}^\pm$.
\end{enumerate}

Step 3: Let $\ul{w}=(R,w_1,...,w_{b-1},w_b)$, $b\geq1$, $l \geq1$. For all $n=b+l>b$:
\begin{equation}\label{draje}
	\wh{p}^{(n)}_{\ul{R}}(\mathbb{1}_{[\ul{w}]})= \sum_{\tigma^b\ul{S}=\ul{R}}e^{\wt{\phi}_b(\ul{S})}\left(\wh{p}_{\ul{S}}^{(l)}\circ \sigma^b\right)(\mathbb{1}_{[\ul{w}]})= \sum_{\tigma^b\ul{S}=\ul{R}}e^{\wt{\phi}_b(\ul{S})}\mathbb{1}_{[\ul{w}]}(\ul{S}).
\end{equation}
This is a constant sequence for all $n> b$.\footnote{Having a consistent explicit formula for the limit measure of shifted cylinders allows us to use it as a definition and apply the Carath\'eodory extension theorem; however we choose to present the argument which uses weak-$*$ limits for future computations.} Therefore, since the cylinders (and their shfits) generate the Borel sigma algebra, and $\wh{p}_{\ul{R}}^{(n_k)}\xrightarrow[k\rightarrow\infty]{\text{weak-}*}\wh{p}_{\ul{R}}$, we get that $\wh{p}_{\ul{R}}^{(n)}\xrightarrow[n\rightarrow\infty]{\text{weak-}*}\wh{p}_{\ul{R}}$; and the limit is independent of the choice of $\ul{S}^\pm\in \tau^{-1}[\{\ul{S}\}]\cap \Tigma$ for a chain $\ul{S}\in \Tigma_L$ in \eqref{greendoor2}. In addition it follows that, $	\forall n\geq0$,
\begin{equation}\label{greendoor3}
	\wh{p}_{\ul{R}}\circ \sigma^{-n}=\sum_{\tigma^n\ul{S}=\ul{R}}e^{\wt{\phi}_n(\ul{S})}\wh{p}_{\ul{S}}.
\end{equation}

\end{proof}
This construction could a-priori depend on the choice of $\psi$; and in the case that $\phi$ is transient, $\psi$ may not be unique. We will later show applications in setups which imply the recurrence of $\phi$, and thus remove the need to choose $\psi$.

\begin{cor}[Absolute Continuity]\label{gauraluilui2}
Let $\phi:\Tigma_L\rightarrow\mathbb{R}$ be a weakly H\"older continuous potential s.t $P_G(\phi)<\infty$, where $\Tigma_L$ is an irreducible locally compact one-sided topological Markov shift.
Let $\psi$ be a $\phi$-harmonic function on $\Tigma_L$. Let $\{\wh{p}_{\ul{R}}\}_{\ul{R}\in\Tigma_L}$ be the measures constructed in 
Theorem \ref{gauraluilui}. Let $[R]\subseteq \Tigma_L$, let $\ul{R},\wt{\ul{R}}\in \Tigma_L\cap [R]$, and let $$\wh{\Gamma}_{\ul{R}\,\wt{\ul{R}}}:\{\ul{v}\in\Tigma: v_i=R_i,\forall i\leq0\}\rightarrow \{\ul{v}\in\Tigma: v_i=\wt{R}_i,\forall i\leq0\}$$ be 
the holonomy map $\ul{v}\mapsto (\wt{R}_i)_{i\leq0}\cdot (v_i)_{i\geq0}$, where $\cdot$ denotes an admissible 
concatenation. Then $\wh{\Gamma}_{\ul{R}\,\wt{\ul{R}}}$ is a continuous 
and invertible map s.t $\wh{p}_{\ul{R}}\circ \wh{\Gamma}_{\ul{R}\,\wt{\ul{R}}}^{-1}\sim\wh{p}_{\wt{\ul{R}}}$. 

\end{cor}
\begin{proof}

Let $\psi$ be the $\phi$-harmonic function which is used in Theorem \ref{gauraluilui}, and write $\wt{\phi}:=\phi+\log\psi-\log\psi\circ\tigma-P_G(\phi)$ and $\wt{\phi}_n:=\sum_{k=0}^{n-1}\wt{\phi}\circ\tigma^k$. By Lemma \ref{logHarmonicHolder} by the local compactness of $\Tigma_L$, the weak H\"older continuity of $\phi$, implies that $\wt{\phi}$ is weakly H\"older continuous.

Let $[\ul{w}]_a:=\sigma^a[w_a,w_{a+1},...,w_{b-1},w_b]$, where $a,b\in\mathbb{Z}$ and $a\leq b$
, and write $g:=\mathbb{1}_{[\ul{w}]_a}$. Let $j \geq0$, and fix an extension $\ul{S}^{\pm,(j)}\in \tau^{-1}[\{\ul{S}\}]$ for each $\ul{S}\in \tigma^{-j}[\{\wt{\ul{R}}\}]$, where $\tau: \Tigma\rightarrow\Tigma_L$ is the projection to the non-positive coordinates. Let $\{\wh{p}_{\ul{R}}\}_{\Tigma_L}$ be the measures constructed in Theorem \ref{gauraluilui}. By steps 2 and 3 in Theorem \ref{gauraluilui}, 
\begin{equation}\label{forVaughnFirst}
\wh{p}_{\wt{\ul{R}}}(g)=\lim_{j\rightarrow\infty}\sum_{\tigma^j\ul{S}=\wt{\ul{R}}}e^{\wt{\phi}_j(\ul{S})}\delta_{\sigma^{-j}\ul{S}^{\pm,(j)}}(g).	
\end{equation}

Moreover, by step 3 in Theorem \ref{gauraluilui}, this limit is independent of the choice of $\{\ul{S}^{\pm,(j)}\}_{j\geq0}$. Thus, similarly,

\begin{equation}\label{forVaughn}
\wh{p}_{\ul{R}}(g\circ\wh{\Gamma}_{\ul{R}\,\wt{\ul{R}}})=\lim_{j\rightarrow\infty}\sum_{\tigma^j\wt{\ul{S}}=\ul{R}}e^{\wt{\phi}_j(\wt{\ul{S}})}\delta_{\wh{\Gamma}_{\ul{R}\,\wt{\ul{R}}}\left(\sigma^{-j}\ul{S}^{\pm,(j)}\right)}(g),
\end{equation}
where $\forall j\geq0$, $\forall \wt{\ul{S}}\in\tigma^{-j}[\{\ul{R}\}]$, $\tau\left(\sigma^{j}\wh{\Gamma}_{\ul{R}\,\wt{\ul{R}}}\left(\sigma^{-j}\ul{S}^{\pm,(j)}\right)\right)=\wt{\ul{S}}$.


Let $C>0,\gamma\in(0,1)$ be the H\"older constant and H\"older exponent 
 of $\wt{\phi}$, respectively. 
$\forall j\geq0$, $\forall \ul{S}\in\tigma^{-j}[\{\wt{\ul{R}}\}]$, write $\wt{\ul{S}}:= \tau\left(\sigma^{j}\wh{\Gamma}_{\ul{R}\,\wt{\ul{R}}}\left(\sigma^{-j}\ul{S}^{\pm,(j)}\right)\right)\in\tigma^{-j}[\{\ul{R}\}]$. Write $d(\ul{R},\wt{\ul{R}})=e^{-n}$, $n\geq1$. Then,
$$e^{\wt{\phi}_j(\ul{S})}=e^{\pm C\sum_{k\geq 0}\gamma^{k+n}}\cdot e^{\wt{\phi}_j(\wt{\ul{S}})}.$$
Write $\wt{C}_\phi:=\exp( C\cdot \sum_{k\geq 0}\gamma^k) \in (0,\infty)$. Then,
\begin{equation}\label{seasalt}
	\forall h\in C(\{\ul{v}\in\Tigma: v_i=\wt{R}_i,\forall i\leq0\}),\text{ } \wh{p}_{\wt{\ul{R}}}(h)=\left(\wt{C}_\phi\right)^{\pm \gamma^n}\cdot \left(\wh{p}_{\ul{R}}\circ\wh{\Gamma}_{\ul{R}\,\wt{\ul{R}}}^{-1}\right)(h).
\end{equation}
\end{proof}

\noindent\textbf{Remark:} The proof of Corollary \ref{gauraluilui} shows something a little stronger than absolute continuity, it shows that $\{\wh{p}_{\ul{R}}\}_{\ul{R}\in\Tigma_L}$ is a continuous family with a uniform modulus of continuity for the holonomies.

\begin{definition}\label{greendoor4}
	Let $\Tigma$ be a maximal irreducible component of $\Sig$, and let $\{\wh{p}_{\ul{R}}\}_{\ul{R}\in \Tigma_L}$ be the family of measures given by Theorem \ref{gauraluilui}, when $P_G(\phi|_{\Tigma_L})<\infty$. Define $\forall \ul{R}\in\Tigma_L$, $\wh{\mu}_{\ul{R}}:=\psi(\ul{R})\cdot \wh{p}_{\ul{R}}$, where $\psi$ is the $\phi$-harmonic function on $\Tigma_L$ which is fixed in Theorem \ref{gauraluilui}.
\end{definition}

\begin{cor}\label{jeo}
	The family of measures $\{\wh{\mu}_{\ul{R}}\}_{\ul{R}\in\Tigma_L}$ from Definition \ref{greendoor4} satisfies: 
	
	$\wh{\mu}_{\ul{R}}\circ \sigma^{-1}=e^{-P_G(\phi)}\sum\limits_{\tigma\ul{S}=\ul{R}}e^{\phi(\ul{S})} \wh{\mu}_{\ul{S}}$. In addition, $\{\wh{\mu}_{\ul{R}}\}_{\ul{R}\in\Tigma_L}$ is an absolutely continuous family of measures.
\end{cor}
The invariance follows from Definition \ref{greendoor4} and from \eqref{greendoor3}. The absolute continuity follows from Corollary \ref{gauraluilui2}.

\begin{cor}\label{endofNoel}
	When $\phi$ is recurrent, the family $\{\wh{\mu}_{\ul{R}}\}_{\ul{R}\in\Tigma_L}$ is the unique (up to scaling) continuous family of measures which satisfies $\wh{\mu}_{\ul{R}}\circ \sigma^{-1}= e^{-P_{G}(\phi)}\sum_{\tigma\ul{S}=\ul{R}}e^{\phi(\ul{S})}\wh{\mu}_{\ul{S}}$.
\end{cor} 
\begin{proof}
Let $p$ and $\psi$ be version of the (unique up to scaling) $\phi$-conformal measure and $\phi$-harmonic function on $\Tigma_L$, respectively. Then, $\wh{\mu}=\int \wh{\mu}_{\ul{R}}dp(\ul{R})$ is a $\sigma$-invariant measure on $\Tigma$. By \cite[Theorem~1]{SarigNR}, the recurrence of $\phi$ implies that $p$ is conservative and finite on cylinders. Thus, $\wh{\mu}$ must be conservative as well. Fix $[R]\subseteq \Tigma_L$, and let $[R,R_{-n+2},...,R_{-1},R]$ be an admissible cylinder. Then,
\begin{align}\label{laurafatigue}
\wh{\mu}([R,R_{-n+2},...,R_{-1},R])=\int \wh{\mu}_{\ul{R}}([R,R_{-n+2},...,R_{-1},R])dp\propto(\psi\cdot p)([R,R_{-n+2},...,R_{-1},R]).\end{align}

This determines $\wh{\mu}_{\ul{R}}$ for $p$-a.e $\ul{R}$ (up to a scaling constant), and since $p$ gives a positive volume to every cylinder (\cite{SarigPR,SarigNR}), the continuity of the family $\{\wh{\mu}_{\ul{R}} \}_{\ul{R}\in\Tigma_L}$ determines the family uniquely.  
\end{proof}


\begin{definition}[Leaf Measures]\label{LeafMeasures}
	Let $\varphi$ be a $\cont$ potential, and let $\Tigma_L$ be a maximal irreducible component of $\Sig_L$, and let $\ul{R}\in\Tigma_L$. The corresponding {\em leaf measure} is $\mu_{\ul{R}}:=\wh{\mu} _{\ul{R}}\circ \wh{\pi}^{-1}$, where $\wh{\mu} _{\ul{R}}$ is 
	as in Definition \ref{greendoor4}.
\end{definition}

\begin{theorem}\label{smoothMeasures}
Let $\varphi:\HWT_\chi\rightarrow \mathbb{R}$ be a 
$\cont$ potential. Let $H_\chi(p)$ be an ergodic homoclinic class of a periodic and $\chi$-hyperbolic point $p$. Then there exists a corresponding family of leaf measures which is $\varphi$-invariant, and is absolutely continuous, that is,
	\begin{enumerate}
		\item Leaf measures: $\forall\ul{R}\in\Tigma_L$, $\mu_{\ul{R}}$ is carried by $\wh{\pi}[\{\ul{v}\in\Tigma:v_i=R_i,\forall i\leq0\}]\subseteq V^u(\ul{R})$. 
		\item $\varphi$-invariance up to a bounded coboundary: $\forall\ul{R}\in\Tigma_L$, $$\mu_{\ul{R}}\circ f^{-1}=e^{-P_{H_\chi(p)}(\varphi)}\cdot\sum_{\tigma\ul{S}=\ul{R}}e^{A-A\circ f^{-1}}\cdot e^{\varphi}\cdot\mu_{\ul{S}},$$ where $A:\HWT_\chi\rightarrow\mathbb{R}$ is bounded and given by Theorem \ref{SinaiBowen}.
		\item Absolute continuity: for any cylinder $ [R]\subseteq \Tigma_L$, and for any two chains $ \ul{R},\ul{S}\in \Tigma_L\cap[R]$, let $$\Gamma_{\ul{R}\,\ul{S}}: \wh{\pi}[\{\ul{v}\in\Tigma:v_i=R_i,\forall i\leq0\}] \rightarrow \wh{\pi}[\{\ul{v}\in\Tigma:v_i=S_i,\forall i\leq0\}] $$ be the holonomy map along the stable leaves of $\wh{\pi}[\{\ul{v}\in\Tigma:v_i=R_i,\forall i\leq0\}]$ (it is well-defined on the rectangle $\wpi[[R]]$). Then $\left(\mu_{\ul{R}}\circ \Gamma_{\ul{R}\,\ul{S}}^{-1}\right)(g)=C_\phi^{\pm1}\cdot \mu_{\ul{S}}(g)$ for every $g\in C(V^u(\ul{S}))$, where $C_\phi>0$ is a global constant depending only on $\phi$ (and a fixed corresponding harmonic function).
	\end{enumerate}
	$\mathcal{F}_{H_\chi(p)}(\varphi):=\{\mu_{\ul{R}}\}_{\ul{R}\in\Tigma_L}$ is a $\varphi$-invariant family of leaf measures (see Definition \ref{pcwis}).
\end{theorem}
\begin{proof}
Let $A,\phi:\HWT_\chi\rightarrow\mathbb{R}$ be given by Theorem \ref{SinaiBowen} s.t $\varphi=\varphi^*+A-A\circ f^{-1}$, where $A$ is bounded and $\phi=\varphi^*\circ\wpi:\Sig\rightarrow\mathbb{R}$ is well defined and weakly H\"older continuous by the remark after Theorem \ref{SinaiBowen}; and $\phi(\ul{R})$ depends only on the negative coordinates of $\ul{R}$. 

Let $\Tigma$ be a maximal irreducible component of $\Sig$ s.t $\wpi[\Tigma^\#]=H_\chi(p)$ modulo all conservative measures (see Proposition \ref{homoclinicirreducible}). Let $[R]\subseteq \Tigma_L$, and let $\ul{R},\ul{S}\in[R]\cap \Tigma_L$. 

Since $\varphi$ is bounded, so is $\phi$.  W.l.o.g $P_G(\phi|_{\Tigma_L})=0$, otherwise replace $\phi$ by $\phi-P_G(\phi|_{\Tigma_L})$. By the spectral decomposition, we may assume w.l.o.g that $(\Tigma,\sigma)$ is topologically mixing (see \cite[Theorem~2.5]{SarigTDFSymposium}). Thus, by the variational principle see \cite[Theorem~4.4]{SarigTDF} (and the fact that $\wpi|_{\Tigma^\#}$ is finite-to-one which allows to lift probabilities while preserving the entropy\footnote{Every invariant probability measure can be decomposed into its ergodic components, and each component lifts to an invariant probability measure. Then this lift is $N$-to-one for some $N\in\mathbb{N}$ since the cardinality of fibers is an invariant function. Finite extensions preserve entropy, and entropy is affine.}), $P_G(\phi|_{\Tigma})=P_{H_\chi(p)}(\varphi)<\infty$.
\begin{enumerate}
	\item \begin{align*}
\mu_{\ul{R}}(\wh{\pi}[\{\ul{v}\in\Tigma:v_i=R_i,\forall i\leq0\}])= &\wh{\mu}_{\ul{R}}\circ\wh{\pi}^{-1}(\wh{\pi}[\{\ul{v}\in\Tigma:v_i=R_i,\forall i\leq0\}])\\ 
\geq &\wh{\mu}_{\ul{R}}(\{\ul{v}\in\Tigma:v_i=R_i,\forall i\leq0\})=\wh{\mu}_{\ul{R}}(1)= \mu_{\ul{R}}(1). 	
 \end{align*}
 So $\wpi[\{\ul{v}\in \Tigma: v_i=R_i, \forall i\leq 0\}] \subseteq V^u(\ul{R})$ has full measure.
	\item By the relation $\wh{\pi}\circ\sigma=f\circ\wh{\pi}$ and Corollary \ref{jeo}, 
\begin{align*}
	\mu_{\ul{R}}\circ f^{-1}= \wh{\mu}_{\ul{R}}\circ\wh{\pi}	^{-1}\circ f^{-1}=\left(\wh{\mu}_{\ul{R}}\circ 	\sigma^{-1}\right)\circ \wh{\pi}^{-1}= \left(\sum_{\tigma\ul{Q}=\ul{R}}e^{\phi(\ul{Q})}\wh{\mu}_{\ul{Q}}\right)\circ \wh{\pi}^{-1}= \sum_{\tigma\ul{Q}=\ul{R}}e^{\phi(\ul{Q})}\mu_{\ul{Q}}.
\end{align*}
For a continuous test function $g\in C(V^u(\ul{Q}))$,
\begin{align*}
	e^{\phi}\cdot \mu_{\ul{Q}}(g)=\int\limits_{\substack{\ul{Q}'\in\Tigma:\\ Q_i=Q_i',\forall i\leq0}} e^{\varphi^*\circ\wpi}g\circ\wpi d\wh{\mu}_{\ul{Q}}= \int\limits_{\substack{\ul{Q}'\in\Tigma: \\ Q_i=Q_i',\forall i\leq0}} \left(e^{\varphi^*}g\right)\circ \wpi d\wh{\mu}_{\ul{Q}}=\left(e^{\varphi+A-A\circ f^{-1}}\cdot \mu_{\ul{Q}}\right)(g).
\end{align*}
This completes the $\varphi$-invariance up to a bounded coboundary.
	\item Let $\Gamma_{\ul{R}\,\ul{S}}: \wh{\pi}[\{\ul{v}\in\Tigma:v_i=R_i,\forall i\leq0\}] \rightarrow \wh{\pi}[\{\ul{v}\in\Tigma:v_i=S_i,\forall i\leq0\}] $ be the holonomy map along the stable leaves of points in $\wh{\pi}[\{\ul{v}\in\Tigma:v_i=R_i,\forall i\leq0\}]$. Let $\wh{\Gamma}_{\ul{R}\,\ul{S}}:\{\ul{v}\in\Tigma: v_i=R_i,\forall i\leq0\}\rightarrow \{\ul{v}\in\Tigma: v_i=S_i,\forall i\leq0\}$ be a continuous and invertible map s.t $\ul{v}\mapsto (S_i)_{i\leq0}\cdot (v_i)_{i\geq0}$, where $\cdot$ denotes an admissible concatenation. Let $g\in C(V^u(\ul{S}))$, and write $\wh{g}:=g\circ \wh{\pi}\in C(\{\ul{v}\in\Tigma:v_i=S_i,\forall i\leq0\})$. It follows that $\Gamma_{\ul{R}\,\ul{S}}\circ\wh{\pi}=\wh{\pi}\circ \wh{\Gamma}_{\ul{R}\,\ul{S}}$, and so $\left(\mu_{\ul{R}}\circ \Gamma_{\ul{R}\,\ul{S}}^{-1}\right)(g)=\left(\wh{\mu}_{\ul{R}}\circ \wh{\Gamma}_{\ul{R}\,\ul{S}}^{-1}\right)(\wh{g})= \psi(\ul{R})\cdot\left(\wh{p}_{\ul{R}}\circ \wh{\Gamma}_{\ul{R}\,\ul{S}}^{-1}\right)(\wh{g}) $ and $\mu_{\ul{S}}(g)= \psi(\ul{S})\cdot\wh{p}_{\ul{S}}(\wh{g})$. By Lemma \ref{logHarmonicHolder}, $\mathrm{Var}_1(\log\psi)<\infty$. Thus, by \eqref{seasalt}, we are done.
\end{enumerate}	
\end{proof}

\noindent\textbf{Remark:} Absolute continuity implies a local product structure. Other related important properties of hyperbolic equilibrium states of H\"older continuous potentials are true, such as being Bernoulli up to a period. This has been proven when $\dim M=2$ in \cite{EqStateBernoulli} and for the setup of countable Markov shifts in \cite{Daon}. Sarig's proof extends, as it only uses the fact that the equilibrium state can be coded as an equilibrium state on a countable Markov shift, of a weakly H\"older continuous potential. In \cite{Daon} Daon relaxes the assumption on the regularity of the potential to the Walters property.

\section{Proofs of Main Results}
\subsection{Local Equilibrium States and the Leaf Condition}

\begin{theorem}\label{main1WIS}
	Let $M$ be a compact Riemannian manifold without boundary and of dimension $d\geq2$. Fix $\chi>0$ and $\epsilon=\epsilon_\chi$ as in Lemma \ref{forChaptoro6ixo}. Let $f\in \mathrm{Diff}^{1+\beta}(M)$, $\beta>0$, let $\chi>0$, and let $\varphi:M\rightarrow \mathbb{R}$ be a 
$\cont$ potential. Let $p$ be a $\chi$-hyperbolic periodic point. Then, there exists a $\chi$-hyperbolic local equilibrium state of $\varphi$ on $H_\chi(p)$ if and only if the following two conditions hold:
	\begin{enumerate}
		\item $$\sum_{n\geq1}\sum_{f^n(q)=q,q\in H_\chi(p)\cap \Lambda_l}e^{\varphi_n(q)-n\cdot P_{H_\chi(p)}(\varphi)}=\infty,$$
where $\Lambda_l$ is a level set s.t $p\in \Lambda_{l\cdot e^{-2\sqrt[3]{\epsilon}}}$; and 
		\item there exists a leaf measure 
$\mu\in\mathcal{F}_{H_\chi(p)}(\varphi)$ s.t $$\mu(\HWT_\chi^{\mathrm{PR}})>0,$$
	where $\mathcal{F}_{H_\chi(p)}(\varphi)$ is any $\varphi$-invariant family of measures given by Theorem \ref{smoothMeasures}.
\end{enumerate}
If a $\chi$-hyperbolic local equilibrium state of $\varphi$ exists on $H_\chi(p)$, then its conditional measures on unstable leaves are proportional to the members of the $\varphi$-invariant family of measures; and it is unique. 
\end{theorem}
\begin{proof}
The ``only if" direction is straightforward: Let $\mu$ be a $\chi$-hyperbolic local equilibrium state of $\varphi$ on $H_\chi(p)$. Let $\Tigma$ be a maximal irreducible component which lifts all conservative measures on $H_\chi(p)$ (see Proposition \ref{homoclinicirreducible}). Lift $\mu$ to an ergodic, $\sigma$-invariant, probability measure on $\Tigma$, $\wh{\mu}$, s.t $\wh{\mu}\circ\wh{\pi}^{-1}=\mu$, and $h_{\wh{\mu}}(\sigma)+\int \varphi\circ\wpi d\wh{\mu}= h_\mu(f)+\int \varphi d\mu$. This is possible by choosing a suitable ergodic component of $\int \frac{1}{|N_x|}\sum_{\ul{v}\in N_x}\delta_{\ul{v}}d\mu(x)$, where $N_x:=\wh{\pi}^{-1}[\{x\}]\cap\Sig^\#$. Let $\phi:\Tigma_L\rightarrow\mathbb{R}$ be given by Theorem \ref{SinaiBowen} and the remark which follows it, and let $\{\wh{\mu}_{\ul{R}}\}_{\ul{R}\in\Tigma_L}$ be given by Definition \ref{greendoor4}. 

By the spectral decomposition (see \cite[Theorem~2.5]{SarigTDFSymposium}), we may assume w.l.o.g that $(\Tigma_L,\tigma) $ is topologically mixing. It follows that 
\begin{align*}
h_{\wh{\mu}}(\sigma)+\int \phi d\wh{\mu}=&\sup\{h_{\wh{\nu}}(\sigma)+\int \phi d\wh{\nu}:\wh{\nu}\text{ is an invariant probability measure on }\Tigma\}\\
= &\sup\{h_{\wh{\nu}\circ\tau^{-1}}(\tigma)
+\int \phi d\wh{\nu}\circ\tau^{-1}:\wh{\nu}\text{ is an invariant probability measure on }\Tigma\},
 \end{align*}
 where $\tau:\Tigma\rightarrow\Tigma_L$ is the projection onto the non-positive coordinates. Every invariant probability $\nu$ on $\Tigma_L$ can be lifted to an invariant probability measure on $\Tigma$, $\wh{\nu}$, s.t $\wh{\nu}\circ \tau^{-1}=\nu$ and $h_{\nu}(\tigma)=h_{\wh{\nu}}(\sigma)$. So, 
 $$h_{\wh{\mu}\circ\tau^{-1}}(\tigma)+\int \phi d\wh{\mu}\circ\tau^{-1}=\sup\{h_{\nu}(\sigma)+\int \phi d\nu:\nu\text{ is an invariant probability measure on }\Tigma_L\}.$$

 Therefore, by \cite{BuzziSarig}, $\phi$ is recurrent on $\Tigma_L$. Thus, by Lemma \ref{periodicRecurrence}, $\sum\limits_{n\geq1}\sum\limits_{\substack{f^n(q)=q,\\ q\in H_\chi(p)\cap \Lambda_l}}e^{\varphi_n(q)-n\cdot P_{H_\chi(p)}(\varphi)}=\infty$. 

By \cite[Theorem~4.5]{SarigTDF}, $\wh{\mu}\circ\tau^{-1}=\psi\cdot p$, where $\psi$ is the $\phi$-harmonic function and $p$ is the $\phi$-conformal measure, scaled s.t $\psi\cdot p(1)=1$, and $\psi,p$ are unique up to a scaling. Then, by \eqref{laurafatigue}, $\wh{\mu}=\int_{\Tigma_L} \wh{\mu}_{\ul{R}}dp(\ul{R})$
. Thus, $\mu=\int \mu_{\ul{R}}dp$, where $\{\mu_{\ul{S}}\}_{\ul{S}\in\Tigma_L}$ is the $\varphi$-invariant family of leaf measures as in Theorem \ref{smoothMeasures}. In addition, $\mu$ is carried by $\HWT_\chi^{\mathrm{PR}}$ since it is an invariant probability measure. This concludes the ``only if" direction.

\medskip
The ``if" direction: $\phi$ is recurrent on $\Tigma_L$ by Lemma \ref{periodicRecurrence}. Let $p$ be a $\phi$-conformal measure carried by $\Tigma_L^\#$, and let $\{\mu_{\ul{S}}\}_{\ul{S}\in\Tigma_L}$ be the $\phi$-invariant family of leaf measures given by Theorem \ref{smoothMeasures}. By \cite[Theorem~1]{SarigNR}, $p$ is ergodic and conservative. Let
$$\mu:=\int_{\Tigma_L}\mu_{\ul{R}}dp(\ul{R}).$$ 
It follows that $\mu$ is an ergodic, $f$-invariant, and conservative measure on $H_\chi(p)$, and that $\mu$ is finite on Pesin level sets. This measure is finite if and only if $\phi$ is positive recurrent. For full details of the proof of these properties, see \cite[Theorem~6.2]{SRBleaf}. (In that proof $\phi$ denotes the geometric potential, but the proof is identical for general weakly H\"older continuous functions).

Notice that $\mu(1)=p(\psi)$, where $\psi$ is the unique (up to scaling) $\phi$-harmonic function on $\Tigma_L$. By \cite[Theorem~8]{SarigPR} if $h_{\mathrm{top}}(\sigma),\infty$ then $\mu(1)<\infty$ if and only if $\psi\cdot p$ is an equilibrium state of $\phi$ on $\Tigma_L$ (and thus $\mu$ is a local $\chi$-hyperbolic equilibrium state of $\varphi$ on $H_\chi(p)$).

By assumption $\exists \ul{R}\in \Tigma_L$ s.t $\mu_{\ul{R}}(\HWT_\chi^{\mathrm{PR}})>0$. By \cite[Theorem~6.8]{SRBleaf}, $\mu$ is finite. This concludes the ``if" direction.

The proof shows that the local equilibrium state can always be disintegrated into conditional measures on unstable leaves, where the conditional measures are $\{\mu_{\ul{R}}\}_{\ul{R}\in\Tigma_L}$.

We proceed to show the uniqueness of the $\chi$-hyperbolic local equilibrium state, when it exists. We saw that $\mu=\wh{\mu}\circ\wpi^{-1}$ where $\wh{\mu}$ is an equilibrium state for $\phi$ on $\Tigma$. By \cite{BuzziSarig}, $\wh{\mu}$ is unique and satisfies a variational principle on $\Tigma$. It follows that a $\chi$-hyperbolic local equilibrium states is unique on $H_\chi(p)$.
\end{proof}

\noindent\textbf{Remark:} If $\exists p\in \Lambda_\ell$ s.t $\sum_{n\geq 1}\sum_{f^n(q)=q, q\in H_\chi(p)\cap\Lambda_\ell}e^{\varphi_n(q)-n P_{H_\chi(p)}(\varphi)}=\infty$, 
then \linebreak$\sum_{n\geq 1}\sum_{f^n(q)=q, q\in H_\chi(p)\cap\Lambda_{\ell e^{2\sqrt[3]\epsilon}}}e^{\varphi_n(q)-n P_{H_\chi(p)}(\varphi)}=\infty$. Similarly, if $p\in \Lambda_{\ell e^{-2\sqrt[3]\epsilon}}$, then $p\in \Lambda_{\ell}$. Then the condition in Theorem \ref{main1WIS}(2) can be relaxed so $p\in \Lambda_l$.

\subsection{Ledrappier-Young Property for Hyperbolic Equilibrium States}
In the celebrated results of \cite{LedrappierYoungI, LedrappierYoungII}, Ledrappier and Young prove a general formula for the entropy of invariant probability measures of diffeomorphisms in terms of the local dimensions of their conditional measures. One very important application of their theory, is the characterization of SRB measures as those which satisfy Pesin's entropy formula, extending the previous same result of Ledrappier for hyperbolic SRB measures \cite{Ledrappier}. That is, an ergodic and hyperbolic invariant probability measure, $\mu$, which satisfies the following two properties must have conditional measures on unstable local manifolds which are absolutely continuous w.r.t the induced Riemannian volume measure: Let $\varphi(x):= -\log\Jac(d_xf|_{H^u(x)}):\HWT_\chi\rightarrow\mathbb{R}$ (the geometric potential), then
\begin{enumerate}
	\item $$h_{\mu}(f)+\int \varphi d\mu=\sup\Big\{h_{\nu}(f)+\int \varphi d\nu: \nu\text{ is an } \begin{array}{ll} \text{erg.}\\ \text{hyp.}\\ f\text{-inv.}\end{array}\text{prob.}\Big\},$$ where the geometric potential in the integrand is well defined almost everywhere for every hyperbolic invariant probability measure,
	\item $\sup\left\{h_{\nu}(f)+\int\varphi d\nu: \nu\text{ is an erg. hyp. } f\text{-inv. prob.}\right\}=0$,
\end{enumerate}
Note that $P_G(\phi|_{\Tigma_L})\leq0$ is always true, because of the Ruelle-Margulis inequality \cite{MargulisRuelleIneq}. The fact that SRB measures satisfy (1),(2) is due to Ledrappier and Strelcyn \cite{EntropyFormulaSRB}. 
A different way to write this characterization, is the following: Let $\mu$ be an ergodic hyperbolic invariant probability measure. Then $\mu$ is an SRB measure if and only if, 
\begin{enumerate}
	\item $\phi$ is positive recurrent when lifted to an irreducible component $\Tigma_L\subseteq \Sig_L$ (see \textsection \ref{rec1}), where $\phi$ is given by Theorem \ref{SinaiBowen} applied to $\varphi(x)
$,
	\item $P_G(\phi|_{\Tigma_L})=0$.
\end{enumerate}	
It follows that conditions (1) and (2) above are satisfied $\Leftrightarrow$ $\mu_{V^u}\ll m_{V^u}$, where $\mu_{V^u}$ is a conditional measure of $\mu$ w.r.t a measurable partition of local unstable leaves $\{V^u\}$, and $m_{V^u}$ is the induced Riemannian volume of $V^u$ (the conformal measure of the geometric potential $\varphi$ on $V^u$).

In this section, we extend this result by replacing the geometric potential $\varphi$ by a general potential which is 
$\cont$
, and by replacing the induced Riemannian volume $m_{V^u}$ 
by any $\varphi$-conformal family. 
We obtain a new proof 
 different to that of Ledrappier in the hyperbolic case \cite{Ledrappier}, and to that of Ledrappier and Young in the general case \cite{LedrappierYoungI}.

\begin{theorem}\label{mainWIS2}
	Let $M$ be a compact Riemannian manifold without boundary, and of dimension $d\geq2$. Let $f\in \mathrm{Diff}^{1+\beta}(M)$, $\beta>0$, let $\chi>0$ and $\epsilon=\epsilon_\chi>0$ as in Lemma \ref{forChaptoro6ixo}. Let $\varphi:WT_\chi^\epsilon\rightarrow \mathbb{R}$ be a 
$\cont$ potential (recall Definition \ref{MnfldWeakHolder}). Let $q$ be a $\chi$-hyperbolic periodic point. Assume that $H_\chi(q)$ admits a (unique) $\chi$-hyperbolic equilibrium state of $\varphi$, $\nu$
. For any family of conditional measures 
as in Theorem \ref{main1WIS}
, $\mathcal{F}_{H_\chi(q)}(\varphi)$, there is a sub-family $\mathcal{F}'$ which disintegrates $\mu$ as in Theorem \ref{main1WIS}, and for which the following holds: Fix $\mu\in \mathcal{F}
'$ which is carried by a local unstable leaf $V^u$. Assume that $H_\chi(q)$ admits a $\varphi$-conformal system of measures $\mathcal{C}:=\{m^\varphi_{W^u}:W^u\text{ is a local unstable leaf of }H_\chi(q)\}$.
Then, $\mu\ll m^\varphi_{V^u}$. In particular, $m_{V^u}^\varphi$ gives a positive measure to $\HWT_\chi^{\mathrm{PR}}$.
\end{theorem}
\begin{proof}
Let $\phi:\HWT_\chi\rightarrow\mathbb{R}$ be given by Theorem \ref{SinaiBowen}. As in the proof of Theorem \ref{smoothMeasures}, let $\Tigma$ be a maximal irreducible component of $\Sig$ s.t $\wpi[\Tigma^\#]=H_\chi(q)$ modulo conservative measures; and write $\mathcal{F}_{H_\chi(q)}(\varphi)=\{\mu_{\ul{R}}\}_{\ul{R}\in\Tigma_L}$. 

Denote the lift of $\nu$ to $\Tigma$ by $\wh{\nu}$. As in the proof of Theorem \ref{main1WIS}, $\phi$ is positive recurrent on $\Tigma$ and $\wh{\nu}=\int_{\Tigma_L}\wh{\mu}_{\ul{R}}dp(\ul{R})$, where $p$ is the unique (up to scaling) $\phi$-conformal measure on $\Tigma_L$. In \cite{SarigPR}, Sarig showed that when $\phi$ is positive recurrent, $p$ is conservative, and is carried by $\Tigma_L^\#$. 
Set $
\mathcal{F}':=\{\mu_{\ul{R}}\}_{\ul{R}\in\Tigma_L^\#}$. 
Write $\mu=\mu_{\ul{R}}$, where $\ul{R}\in\Tigma_L^\#$. Let $S\in\mathcal{R}$ be a symbol which repeats infinitely often in $\ul{R}$. 

Let $K\in\mathbb{N}$ be $N(R_0)\cdot N(S)$ (see Definition \ref{N_R}). This is a bound on the number of sequences in $\Sig^\#$ which code the same point and s.t this point can be coded by a sequence which has $R_0$ repeat infinitely often in the future and $S$ repeat infinitely often in the past (see \cite[Theorem~1.3]{SBO}). Choose 
$n\geq 1$ and $\ul{S},\ul{Q}\in\Tigma_L$ s.t $\sigma_R^n\ul{S}=\sigma_R^n\ul{Q}=\ul{R}
$ and $S_0=Q_0=R_0$. Let $\ul{S}^\pm,\ul{Q}^\pm\in\Tigma
$ 
which return 
to $[R_0]$ infinitely often in the future 
and s.t $S^\pm_i=S_i,\forall i\leq0$ and $Q^\pm_i=Q_i,
\forall i\leq0$. Define $x:=\wpi(\ul{S}^\pm),y:=\wpi(\ul{Q}^\pm),z:=\wpi([\ul{S}^\pm,\ul{Q}^\pm]_{R_0})$, where  $[\ul{S}^\pm,\ul{Q}^\pm]_R$ is the Smale bracket of $\ul{S}^\pm$ and $
\ul{Q}^\pm$ in $[R_0]$. If $V^u(\ul{S})\cap V^u(\ul{Q})
\neq\varnothing$, then by \cite[Proposition~3.12]{SBO} $x=z=y$. 

Then for each $\underline{S}\in\tigma^{-n}[\{\ul{R}\}]\cap[R_0]$, the number of $\underline{Q}\in \tigma^{-n}[\{\ul{R}\}]\cap[R_0]$ s.t $V^u(\underline{S})\cap V^u(\ul{Q})\neq\varnothing$ is bounded by $K$. That is,

$$\#\{\ul{S}'\in\Tigma_L:S'_0=R_0,\sigma_R^n\ul{S}'=\ul{R}, V^u(\ul{S}')\cap V^u(\ul{S})\neq\varnothing\}\leq K.$$ 
Define on $V^u(\ul{R})$ the density function $\rho_n:= \sum_{\tigma^n\ul{S}'=\ul{R},S'_0=R}
\mathbb{1}_{V^u(\ul{S}')}\circ f^n$, then for any $n\geq1$,

\begin{equation}\label{forVolLemma}
0\leq \rho_n\leq K.
\end{equation}

Let $c_{R_0}:=\min\{\sup \{m_{V^u(\ul{R}')}^\varphi(1):\ul{R}'\in[R_0]\}^{-1},\inf\{m_{V^u(\ul{R}')}^\varphi(1):\ul{R}'\in[R_0]\}\}$. This is positive by the continuity of $\mathcal{C}$, since $[R_0]$ is compact and $\ul{R}'\mapsto V^u(\ul{R}')$ is continuous in $C^1$-norm (see Definition \ref{Cara}).

Let $g\in C(V^u(\ul{R}))$ s.t $g$ is $L$-Lipschitz, and $\|g\|_\infty\leq 1$. Write $\wh{g}:=g\circ \wpi$, and so $\mu_{\ul{R}}(g)=\wh{\mu}_{\ul{R}}(\wh{g})$. Let 

\begin{equation}\label{weakeromega}
\wh{\omega}_n:=\frac{1}{n}\sum_{k=0}^{n-1}\mathbb{1}_{[R_0]}\circ \sigma^k.	
\end{equation}
Notice, $|\wh{\omega}_n|\leq 1$. By Birkhoff's ergodic theorem, and the absolute continuity of $\{\wh{\mu}_{\ul{R}}\}_{\ul{R}\in\Tigma}$ w.r.t holonomies (see Corollary \ref{gauraluilui2}), for $\wh{\mu}_{\ul{R}}$-a.e $\ul{R}^\pm$, $\wh{\omega}_n(\ul{R}^\pm)\xrightarrow[n\rightarrow\infty]{}\wh{\nu}([R_0])$, and so $\wh{g}(\ul{R}^\pm)\cdot\wh{\omega}_n(\ul{R}^\pm)\xrightarrow[n\rightarrow\infty]{}\wh{\nu}([R_0])\cdot\wh{g}(\ul{R}^\pm)$. Then, by Lebesgue's dominated convergence theorem, 

\begin{equation}\label{weakerMuHatLimit}
\wh{\mu}_{\ul{R}}(\wh{g}\cdot \wh{\omega}_n)\xrightarrow[n\rightarrow\infty]{}\wh{\nu}([R_0])\cdot \wh{\mu}_{\ul{R}}(\wh{g})= \wh{\nu}([R_0])\cdot \mu_{\ul{R}}(g).	
\end{equation}

Since $\wpi$ is H\"older continuous, and $g$ is $L$-Lipschitz, $\wh{g}$ is H\"older continuous as well. Let $H_L>0, \theta\in (0,1)$ s.t $d(\ul{R}^\pm,\ul{S}^\pm)\leq e^{-n}\Rightarrow |\wh{g}(\ul{R}^\pm)-\wh{g}(\ul{S}^\pm)|\leq H_L\cdot \theta^n$.

For every $ n\geq 1$ and $ \ul{S}\in\Tigma_L$ s.t $\sigma_R^n\ul{S}=\ul{R}$, fix $\ul{S}^\pm\in \Tigma$ s.t $S^\pm_i=S_i$ $\forall i\leq 0$. By \cite[Proposition~4.4]{SBO}, 
\begin{equation}\label{weakercontraction}
\mathrm{diam}_{V^u(\ul{R})}(f^{-n}[V^u(\ul{S})])\leq 4e^{-\frac{\chi\cdot n}{2}},
\end{equation}
 where $\mathrm{diam}_{V^u(\ul{R})}$ denotes the diameter w.r.t the induced Riemannian metric on $V^u(\ul{R})$.

\medskip
Let $\underline{W}=(R_0,W_1,\ldots,W_{n-2},R_0)$ be an admissible word of length $n\geq 1$. We estimate $\wh{\mu}_{\ul{R}}([\underline{W}])$: First 
we write $[\underline{W}]= \sigma^{-n}\sigma^n[\underline{W}] $. Next, 
by Corollary \ref{endofNoel}, 
\begin{align*}
\wh{\mu}_{\underline{R}}([\underline{W}]) = & \wh{\mu}_{\underline{R}}(\sigma^{-n}\sigma^n[\underline{W}])=
e^{-n P_{G}(\phi)}\sum_{\tigma^n\ul{S}=\ul{R}}e^{\phi_n(\ul{S})}\wh{\mu}_{\underline{S}} (\sigma^n[\underline{W}])
= 
e^{-n P_{G}(\phi)+\phi_n(\ul{R}
\cdot \ul{W})}\wh{\mu}_{\underline{R}
\cdot\underline{W}} (1)
,
\end{align*}
where $\cdot$ denotes an admissible concatenation. Recall that $\wh{\mu}_{\ul{R}'}(1)=\psi(\ul{R}')$, where $\psi$ is the unique (up to scaling) $\phi$-harmonic function on $\Tigma_L$. Then for 
$C_{R_0}:=\max_{[R_0]}\{\psi,\psi^{-1}\}$
, 
\begin{equation}\label{eqNoel}
	\wh{\mu}_{\underline{R}}([\underline{W}])=C_{R_0}^{\pm 1}\cdot e^{-n P_{G}(\phi)+\phi_n(\ul{R}\cdot \ul{W})}.
\end{equation}

\medskip
Since $\mathcal{C}$ is a $\varphi$-conformal family, we get that for all $n\geq 1$ and for every $ \ul{S}\in\tigma^{-n}[\{\ul{R}\}]$, $$m^\varphi_{V^u(\ul{R})}(f^{-n}[V^u(\ul{S})])= \int e^{\varphi_n(x)-n P_{H_\chi(q)}(\varphi)}dm^\varphi_{V^u(\ul{S})}.$$
We wish to estimate $\varphi_n(\cdot)$ on $V^u(\ul{S})$. For that we need the non-trivial fact that $d_{\mathrm{TM}}(f^{-i}(x),f^{-i}(y))$ decreases exponentially fast in $i\geq0$. This is true since $V^u(\ul{S})$ is contained in the graph of a function with a H\"older continuous derivative (and $d(f^{-i}(x),f^{-i}(y))\leq 4e^{-\frac{\chi}{2}i}$), and the H\"older constant and exponent do not depend on the choice of $\underline{S}\in[R_0]$. For more details see \cite[Definition~3.1]{SBO}. Thus, since $\varphi$ is $\cont$, we get that $\exists C_\varphi>0$ s.t $\forall x,y\in V^u(\ul{S})$, $\sup_{n\geq0}|\varphi_n(x)-\varphi_n(y)|\leq C_\varphi$. Using the fact that $V^u(\ul{S})$ contains a point with a coding in $\Tigma^\#$, together with Theorem \ref{SinaiBowen},  where $\varphi=\phi+A-A\circ f^{-1}$ with $\|A\|_\infty<\infty$, we get for all $n\geq1$,
\begin{equation}\label{eqNoel3}
	m^\varphi_{V^u(\ul{R})}(f^{-n}[V^u(\ul{S})])=e^{\pm (C_\varphi+2\|A\|)}e^{\phi_n(\underline{S})-nP_{H_\chi(q)}(\varphi)}\cdot m_{V^u(\ul{S})}(1)= (c_{R_0}^{-1}e^{C_\varphi+2\|A\|})^{\pm1}e^{\phi_n(\underline{S})-nP_G(\phi)}.
\end{equation}
In the last equality we used the fact that $P_{G}(\phi)=P_{H_\chi(q)}(\varphi)$ (recall \eqref{underpressure2013}).

\medskip
The last identity we need before the main computation of the proof is the following: for any $n\geq1$,
\begin{equation}\label{eqNoel2}
	(\mathbb{1}_{[R_0]}\circ \sigma^n)\cdot\wh{\mu}_{\ul{R}}=\sum_{\substack{|\underline{W}|=n\\ W_0=W_{n-1}=R_0}}\wh{\mu}_{\underline{R}}|_{[\underline{W}]}.
\end{equation}

On the left-hand-side, $\mathbb{1}_{[R_0]}\circ \sigma^n$ acts as a density for $\wh{\mu}_{\ul{R}}$, and on the right-hand-side, $\wh{\mu}_{\underline{R}}|_{[\underline{W}]}$ is the restriction of the measure to the respective cylinder.


\medskip
We are now able to use the estimates and identities from \eqref{forVolLemma}, \eqref{weakercontraction},\eqref{eqNoel}, \eqref{eqNoel3}, and \eqref{eqNoel2} to get $\forall n\geq 1$,

\begin{align}
	K\cdot m_{V^u(\ul{R})}^\varphi(g)\geq&(\rho_n\cdot m_{V^u(\ul{R})}^\varphi)(g)\text{ }(\because \eqref{forVolLemma})\nonumber\\
	(\because g\text{ is Lip},\eqref{weakercontraction})\text{ }=&\pm 4LK\cdot  e^{-\frac{\chi\cdot n}{2}}m_{V^u(\ul{R})}^\varphi(1)
	+ \sum_{\tigma^n\ul{S}=\ul{R},S_0=R_0}m_{V^u(\ul{R})}^\varphi(f^{-n}[V^u(\ul{S})])\cdot g(f^{-n}\circ\wpi(\ul{S}^\pm))\nonumber\\
	(\because \eqref{eqNoel3})\text{ }	=& \pm 4LK\cdot  e^{-\frac{\chi\cdot n}{2}}c_{R_0}^{-1}+(c_{R_0}^{-1}e^{C\varphi+2\|A\|} )^{\pm1}\sum_{\tigma^n\ul{S}=\ul{R},S_0=R_0}e^{\phi_n(\ul{S})-n\cdot P_G(\phi)}\cdot \wh{g}(\sigma^{-n}(\ul{S}^\pm)) \nonumber\\
		(\because \eqref{eqNoel})\text{ } =&\pm 4KL\cdot c_{R_0}^{-1} e^{-\frac{\chi\cdot n}{2}}+ (c_{R_0}^{-1}e^{C\varphi+2\|A\|} C_{R_0})^{\pm1} \sum_{|\ul{W}|=n,W_{n-1}=R_0}\wh{\mu}_{\ul{R}}([\ul{W}])\cdot \wh{g}(\sigma^{-n}(\ul{S}^\pm)) \nonumber\\
	(\because \wh{g}\text{ is Lip}, \eqref{eqNoel2})\text{ } =& \pm 4LKc_{R_0}^{-1}\cdot  e^{-\frac{\chi\cdot n}{2}}\pm e^{C_\varphi+2\|A\|}c_{R_0}^{-1}C_{R_0} \cdot H_L\ \theta^n+ (c_{R_0}^{-1}e^{C\varphi+2\|A\|} C_{R_0})^{\pm1}  \cdot\wh{\mu}_{\ul{R}}(\mathbb{1}_{[R_0]}\circ\sigma^n\cdot\wh{g}).\nonumber
\end{align}
Write $\theta_1:=\max\{\theta,e^{-\frac{\chi}{2}}\} \in(0,1)$, $\wt{C}_{R_0}:=4 e^{C_\varphi+2\|A\|}c_{R_0}^{-1}C_{R_0}\cdot \sum_{k\geq0}\theta_1^k<\infty$. Thus, $\forall n\geq 1$,

\begin{equation}\label{trains}
	K\cdot m_{V^u(\ul{R})}^\varphi(g)\geq 2LK\cdot H_L \cdot \wt{C}_{R_0}+\wt{C}_{R_0}^{\pm1}\cdot\wh{\mu}_{\ul{R}}(\mathbb{1}_{[R_0]}\circ\sigma^n\cdot \wh{g}).
\end{equation}

By summing and averaging \eqref{trains} $N$ times, and by \eqref{weakerMuHatLimit} (recall the definition of $\wh{\omega}_N$ in \eqref{weakeromega}),
\begin{align*}
	K\cdot m_{V^u(\ul{R})}^\varphi(g)\geq\frac{1}{N}\sum_{n=1}^{N}(\rho_n\cdot m_{V^u(\ul{R})}^\varphi)(g)= \pm\frac{1}{N}2LK\cdot H_L\cdot\wt{C}_{R_0}+\wt{C}_{R_0}^{\pm1}\cdot\wh{\mu}_{\ul{R}}(\wh{\omega}_N\wh{g})\xrightarrow[N\rightarrow\infty]{} \wh{\nu}([R_0])\wt{C}_{R_0}^{\pm1}\mu_{\ul{R}}(g).
\end{align*}

 Since $\mathrm{Lip}(V^u(\ul{R}))$ is dense in $\|\cdot\|_\infty$-norm in $C(V^u(\ul{R}))$, by the Riesz-Kakutani-Markov representation theorem, $\mu_{\ul{R}}\leq \wh{\nu}([R_0])^{-1}\cdot K\cdot\wt{C}_{R_0}\cdot  m_{V^u(\ul{R})}^\varphi$.
\end{proof}
\noindent\textbf{Remarks:}
\begin{enumerate}
\item In fact, the proof shows that not only $\mu_{\ul{R}}\ll m_{V^u(\ul{R})}^\varphi$, but that $\mu_{\ul{R}}= (\wh{\nu}([R_0])^{-1}\cdot K\cdot\wt{C}_{R_0})^{\pm1} \cdot m_{V^u(\ul{R})}^\varphi|_{E_{\ul{R}}^{R_0}}$, where $E_{\ul{R}}^{R_0}:=\wpi[\{\ul{R}^\pm\in \Tigma:\forall i\leq 0, R^\pm_i=R_i\text{ and }\#\{j\geq0:R^\pm_j=R_0\}=\infty\}] $ carries $\mu_{\ul{R}}$.
\item	In addition, the proof shows that the $\varphi$-conformal family of measures is unique up to equivalence of the leaf measures, when restricted to sets like $E_{\ul{R}}^{R_0}$; and it works for a $\varphi$-conformal family where the transformation law is not exactly $e^{\varphi_n-n\cdot P_{H_\chi(p)}(\varphi)}$, but merely up to a multiplicative constant uniform in $n$.
\end{enumerate}

\section*{Acknowledgements} I would like to thank the Eberly College of Science in the Pennsylvania State University for excellent working conditions. I am grateful to the referee who read this manuscript with care and who provided many helpful suggestions which improved this manuscript and the presentation of the results in it. I would also like to thank Prof. Omri Sarig for reading this manuscript and for his wise and useful input.

\section*{Special Notation}
\begin{tabular}{ll}
    $\WT$ & weakly temperable points (Definition \ref{temperable})\\
    $\HWT$ & recurrently weakly temperable points (Definition \ref{temperable}, Definition \ref{NUHsharp})\\
    $\Sigma$ & infinite-to-one Markov extension (Theorem \ref{mainSBO})\\
    $\Sig$ & finite-to-one Markov extension (Definition \ref{Doomsday})\\
    $\mathcal{R}$ & Markov partiton (Definition \ref{Doomsday})\\
    $\Tigma$ & maximal irreducible component of $\Sig$ (Definition \ref{irreducibility})\\
    $\Sig^\circ$ & itineraries of orbits in the Markov partition (Definition \ref{canonico})\\
    $X_L$, $X\in \{\Sigma,\Sig,\Tigma,\Sigma^\circ\}$ & $\{(x_i)_{i\leq0}: (x_i)_{i\in\mathbb{Z}}\in X\}$\\
        $X^\#$, $X\in \{\Sigma,\Sig,\Tigma\}$ & $\{x\in X:\exists a,b: x_i=a, x_{-j}=b\text{ for infinitely many positive }i\text{ and positive }j\}$\\
         $X^\#_L$, $X\in \{\Sigma,\Sig,\Tigma\}$ & $\{x\in X_L:\exists a: x_{-i}=a\text{ for infinitely many positive }i\}$\\
          $\sigma_R$ & the right-shift (Definition \ref{littlerightshift})\\
    $\tigma$ & the restriction of the right-shift to $\Tigma_L$ (Definition \ref{Ruelleo})\\
\end{tabular}

\bibliographystyle{alpha}
\bibliography{Elphi}
\end{document}